\documentclass[a4paper]{amsart}

\usepackage{varioref}
\usepackage{hyperref}
\hypersetup{colorlinks=true,pageanchor=false,
linkcolor=blue,citecolor=red,urlcolor=red}
\usepackage[all]{hypcap}

\allowdisplaybreaks
\newcommand*{\refh}[2]{\hyperref[#2]{#1~\ref{#2}}}

\newcounter{ourcount}
\usepackage[T1]{fontenc}
\usepackage[utf8x]{inputenc}
\usepackage{yfonts}
\usepackage{dsfont}
\usepackage{amscd,amssymb,amsmath,amsthm,amsfonts}
\usepackage{mathtools}
\usepackage{accents}
\usepackage{graphicx}
\usepackage{mathrsfs}
\usepackage[all,cmtip]{xy}
\usepackage{tikz}
\usepackage{tikz-cd}
\usetikzlibrary{calc,matrix,arrows,decorations.pathmorphing}
\usepackage{calc}
\usepackage[cal=boondox,scr=boondoxo]{mathalfa}
\usepackage{marginnote}
\usepackage{booktabs}
\usepackage{enumitem}
\usepackage{epsfig}
\usepackage{placeins}
\usepackage{contour}
\usepackage{ulem}

\newcommand\be{\begin{equation}}
\newcommand\ee{\end{equation}}

\newtheorem{theorem}{Theorem}[section]

\newtheorem{proposition}[theorem]{Proposition}
\newtheorem{lemma}[theorem]{Lemma}

\theoremstyle{definition}

\theoremstyle{remark}
\newtheorem{remark}[theorem]{Remark}

\DeclareMathSymbol{\widetildesym}{\mathord}{largesymbols}{"65}

\newcommand{\N}{\mathbb{N}}
\newcommand{\Z}{\mathbb{Z}}

\newcommand{\R}{\mathbb{R}}
\newcommand{\C}{\mathbb{C}}

\newcommand{\rmt}{\mathrm{t}}

\newcommand{\rmB}{\mathrm{B}}

\newcommand{\rmH}{\mathrm{H}}

\newcommand{\rmL}{\mathrm{L}}

\newcommand{\rmR}{\mathrm{R}}

\newcommand{\rmV}{\mathrm{V}}
\newcommand{\rmX}{\mathrm{X}}

\newcommand{\bbA}{\mathbb{A}}
\newcommand{\bbB}{\mathbb{B}}

\newcommand{\bbI}{\mathbb{I}}

\newcommand{\bbM}{\mathbb{M}}
\newcommand{\bbN}{\mathbb{N}}

\DeclareRobustCommand{\bbSigma}{\mathbin{\text{\includegraphics[height=\heightof{$\mathbf{\Sigma}$}]{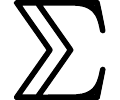}}}}

\newcommand{\calC}{\mathcal{C}}
\newcommand{\calD}{\mathcal{D}}

\newcommand{\calH}{\mathcal{H}}

\newcommand{\calR}{\mathcal{R}}
\newcommand{\calS}{\mathcal{S}}
\newcommand{\calT}{\mathcal{T}}

\newcommand{\calV}{\mathcal{V}}
\newcommand{\calX}{\mathcal{X}}

\newcommand{\frakS}{\mathfrak{S}}

\renewcommand{\epsilon}{\varepsilon}
\renewcommand{\theta}{\vartheta}
\renewcommand{\phi}{\varphi}
\renewcommand{\Gamma}{\varGamma}
\renewcommand{\Sigma}{\varSigma}

\newcommand{\id}{\mathrm{id}}

\newcommand{\ptr}{\mathrm{ptr}}

\newcommand{\ad}{\operatorname{ad}}
\newcommand{\coad}{\operatorname{coad}}
\newcommand{\lrad}{\operatorname{rad_L}}
\newcommand{\rrad}{\operatorname{rad_R}}

\newcommand{\lev}{\smash{\stackrel{\leftarrow}{\mathrm{ev}}}}
\newcommand{\lcoev}{\smash{\stackrel{\longleftarrow}{\mathrm{coev}}}}
\newcommand{\rev}{\smash{\stackrel{\rightarrow}{\mathrm{ev}}}}
\newcommand{\rcoev}{\smash{\stackrel{\longrightarrow}{\mathrm{coev}}}}

\newcommand{\eviso}[1]{\cap_{#1}}
\newcommand{\coeviso}[1]{\cup_{#1}}

\DeclareMathOperator{\disjun}{\sqcup}
\DeclareMathOperator{\din}{\dot{\Rightarrow}}

\renewcommand{\leq}{\leqslant}
\renewcommand{\geq}{\geqslant}

\newcommand{\mods}[1]{\operatorname{\mathnormal{#1}-mod}} 
\newcommand{\bimods}[1]{\operatorname{\mathnormal{#1}-bimod}}

\renewcommand{\sl}{\mathfrak{sl}}
\newcommand{\SL}{\mathrm{SL}}
\newcommand{\GL}{\mathrm{GL}}
\newcommand{\PGL}{\mathrm{PGL}}
\newcommand{\Hom}{\mathrm{Hom}}

\newcommand{\End}{\mathrm{End}}

\newcommand{\Vect}{\mathrm{Vect}}

\newcommand{\op}{\mathrm{op}}
\newcommand{\co}{\mathrm{co}}

\DeclareRobustCommand{\one}{\mathbin{\text{\includegraphics[height=\heightof{$\mathbf{1}$}]{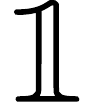}}}}

\newcommand{\Sk}{\mathrm{S}}
\newcommand{\adSk}{\check{\mathrm{S}}}

\newcommand{\Cob}{\mathrm{Cob}}
\newcommand{\adCob}{\check{\mathrm{C}}\mathrm{ob}}

\newcommand{\PB}{\mathrm{PB}}

\newcommand{\Proj}{\mathrm{Proj}}

\makeatletter
\newcommand{\subalign}[1]{
  \vcenter{
    \Let@ \restore@math@cr \default@tag
    \baselineskip\fontdimen10 \scriptfont\tw@
    \advance\baselineskip\fontdimen12 \scriptfont\tw@
    \lineskip\thr@@\fontdimen8 \scriptfont\thr@@
    \lineskiplimit\lineskip
    \ialign{\hfil$\m@th\scriptstyle##$&$\m@th\scriptstyle{}##$\crcr
      #1\crcr
    }
  }
}
\makeatother

\def\clap#1{\hbox to 0pt{\hss#1\hss}}

\newcommand{\pic}[2][0]{\raisebox{-0.5\height + 2.5pt + #1pt}{\includegraphics{#2.pdf}}}

\newcommand{\eend}{\mathcal{E}} 
\newcommand{\coend}{\mathcal{L}} 

\newcommand{\copL}{\Delta}
\newcommand{\counitL}{\epsilon}

\newcommand{\intL}{\Lambda}
\newcommand{\cointL}{\Lambda^{\mathrm{co}}}

\newcommand{\Rad}{\calR}

\newcommand{\Diff}{\mathrm{Diff}}
\newcommand{\eDiff}{\tilde{\mathrm{D}}\mathrm{iff}}
\newcommand{\MCG}{\mathrm{Mod}}
\newcommand{\eMCG}{\tilde{\mathrm{M}}\mathrm{od}}

\newcommand{\repCpr}{\rho_\rmX}
\newcommand{\prepCpr}{\overline\rho_\rmX}
\newcommand{\repL}{\rho_\rmL}
\newcommand{\prepL}{\overline\rho_\rmL}
\newcommand{\repT}{\rho_\calC}
\newcommand{\prepT}{\overline\rho_\calC}
\newcommand{\erepT}{\tilde{\rho}_\calC}

\contourlength{1pt}

\DeclareRobustCommand{\myuline}[1]{
 \ifmmode \text{\uline{$\phantom{#1}$}\llap{\contour{white}{$#1$}}}
 \else \uline{\phantom{#1}}\llap{\contour{white}{#1}} \fi
}
  
\newcommand{\arxiv}[2]{\href{http://arXiv.org/abs/#1}{\texttt{arXiv:#1} #2}}
\newcommand{\doi}[2]{\href{http://doi.org/#1}{#2}}

\begin{document}

\raggedbottom

\title{Mapping Class Group Representations From Non-Semisimple TQFTs}

\author[De Renzi]{Marco De Renzi}
\address{Department of Mathematics, Faculty of Science and Engineering, Waseda University, 3-4-1 \={O}kubo, Shinjuku-ku, Tokyo, 169-8555, Japan} 
\email{m.derenzi@kurenai.waseda.jp}
\address{Institute of Mathematics, University of Zurich, Winterthurerstrasse 190, CH-8057 Zurich, Switzerland}
\email{marco.derenzi@math.uzh.ch}

\author[Gainutdinov]{Azat M. Gainutdinov}
\address{Institut Denis Poisson, CNRS, Universit\'e de Tours, Universit\'e d'Orl\'eans, Parc de Grandmont, 37200 Tours, France}
\address{National Research University Higher School of Economics, Usacheva str., 6, Moscow, Russia}
\email{azat.gainutdinov@lmpt.univ-tours.fr}

\author[Geer]{Nathan Geer}
\address{Mathematics \& Statistics, Utah State University, Logan, Utah 84322, USA} \email{nathan.geer@gmail.com}

\author[Patureau-Mirand]{Bertrand Patureau-Mirand}
\address{Univ. Bretagne - Sud, UMR 6205, LMBA, F-56000 Vannes, France}
\email{bertrand.patureau@univ-ubs.fr}

\author[Runkel]{Ingo Runkel}
\address{Fachbereich Mathematik, Universit\"at Hamburg,
Bundesstra\ss e 55, 20146 Hamburg, Germany}
\email{ingo.runkel@uni-hamburg.de}

\begin{abstract}
 In \cite{DGGPR19}, we constructed 3-dimensional Topological Quantum Field Theories (TQFTs) using not necessarily semisimple modular categories. Here, we study projective representations of mapping class groups of surfaces defined by these TQFTs, and we express the action of a set of generators through the algebraic data of the underlying modular category $\calC$. This allows us to prove that the projective representations induced from the non-semisimple TQFTs of \cite{DGGPR19} are equivalent to those obtained by Lyubashenko via generators and relations in \cite{L94}. Finally, we show that, when~$\calC$ is the category of finite-dimensional representations of the small quantum group of~$\sl_2$, the action of all Dehn twists for surfaces without marked points has infinite order.
\end{abstract}

\maketitle
\setcounter{tocdepth}{2}

\tableofcontents

\date{\today}

\section{Introduction}

The main goal of this paper is to explicitly characterize and compute projective representations of mapping class groups of surfaces coming from the non-semisimple 3-dimensional Topological Quantum Field Theories (TQFTs for short) constructed in \cite{DGGPR19}. In doing this, we establish an equivalence with a family of projective representations constructed by Lyubashenko, while also proving some of their remarkable properties in a concrete example.

\subsection{Background}
In \cite{T94}, Turaev introduced the original semisimple version of modular categories, 
and used them to construct 3-dimensional TQFTs. The latter are defined as symmetric monoidal functors from a category of 3-cobordisms, which carries a symmetric monoidal structure induced by disjoint union, to the category of vector spaces. The theory developed in \cite{T94} is an extension and a categorical reformulation of previous results of Reshetikhin and Turaev \cite{RT90,RT91} which produced topological invariants of links and 3-manifolds out of a particular class of Hopf algebras. Consequently, the TQFTs obtained from 
semisimple modular categories are known as Reshetikhin-Turaev TQFTs.

Around the same time, in \cite{L94} Lyubashenko introduced a more general (and not necessarily semisimple) notion of modular category, and used it to construct 3-manifold invariants and projective representations of mapping class groups of surfaces. These invariants are a categorical reformulation of Hennings' ones \cite{H96}, which are defined in terms of Hopf algebraic data, and which provide the first non-semisimple generalization of the work of Reshetikhin and Turaev. The question whether the theory of \cite{L94} is part of a TQFT, or even of an extended notion of TQFT (known as ETQFT), was addressed in \cite{KL01}. However, the ETQFTs constructed by Kerler and Lyubashenko are defined just for connected surfaces, and are monoidal only in a weaker sense, that is, not with respect to disjoint unions. 

In order to obtain TQFTs from non-semisimple modular categories (in the above sense of symmetric monoidal functors with respect to disjoint union), another ingredient was required: \textit{modified traces} \cite{GPT07,GKP10}. 
Working in the context of Hopf algebras, generalized versions of Hennings' invariants 
were defined with the help of modified traces. 
At first, this was done in the special case of the restricted quantum group of $\sl_2$ \cite{M13,BBG17}. Immediately afterwards, the construction was carried out in the general
Hopf
case, and also upgraded to TQFTs \cite{DGP17}. Finally, the formulation of these TQFTs 
for general modular
categories was given in \cite{DGGPR19}, twenty-five years after Lyubashenko's original work.

\subsection{Non-semisimple TQFTs}

To present our results, we first briefly review the TQFT defined in \cite{DGGPR19}. Let $\calC$ be a modular category in the sense of Section~\ref{sss:MTC}, which means in particular that $\calC$ is not necessarily semisimple. The category $\adCob_\calC$ of \textit{admissible cobordisms} is the symmetric monoidal category having:
\begin{itemize}
    \item \textit{Objects:} surfaces, decorated with finite sets of (oriented framed) marked points labeled by objects of $\calC$ and with Lagrangian subspaces of their first homology group.
    \item \textit{Morphisms:} cobordisms, decorated with embedded bichrome graphs and with integers called \textit{signature defects}, and subject to a crucial \textit{admissibility} condition. A \textit{bichrome graph} is a ribbon graph with red and blue parts. Its blue subgraph, composed of edges and coupons, is labeled by objects and morphisms of $\calC$ respectively, with boundary vertices corresponding to marked points on the boundary of the cobordism. Its red subgraph essentially determines a surgery link, and will be described in more detail in Section~\ref{sss:LRT-functor}. A decorated cobordism is \textit{admissible} if every connected component disjoint from the incoming boundary contains a blue edge labeled by a projective object of $\calC$. 
\end{itemize}
More details can be found in Section~\ref{sss:admissible_cobord}. The category $\adCob_\calC$ has the same objects as the cobordism category used for the Reshetikhin-Turaev TQFT \cite{T94}, but it has fewer morphisms because of the admissibility condition. 
For example, closed 3-manifolds without embedded bichrome graphs are not admissible.
On the other hand, every connected cobordism with non-empty incoming boundary is admissible, in particular every mapping cylinder.

In \cite{DGGPR19} a dual pair of TQFTs, that is, a pair of symmetric monoidal functors
\[
 \rmV_\calC : \adCob_\calC \to \Vect_\Bbbk, \quad 
 \rmV'_\calC : (\adCob_\calC)^{\op} \to \Vect_\Bbbk,
\]
is constructed (see Section~\ref{sss:TQFT-from-universal}). The duality is provided by a non-degenerate pairing $\langle \_,\_ \rangle_{\bbSigma}$ between the state spaces $\rmV'_\calC(\bbSigma)$ and $\rmV_\calC(\bbSigma)$ for each object $\bbSigma$ of $\adCob_\calC$.
The pairing is invariant in the sense that for all morphisms 
$\bbM : \bbSigma \to \bbSigma'$ and for all vectors $v \in \rmV_\calC(\bbSigma)$ and $v' \in \rmV'_\calC(\bbSigma')$ we have
\[
\big\langle v' , \rmV_\calC(\bbM)(v)
\big\rangle_{\bbSigma'} 
=
\big\langle \rmV'_\calC(\bbM)(v') , v \big\rangle_{\bbSigma}
    ~.
\]

These TQFT functors admit both an algebraic description, which is closer to Lyubashenko's approach, and a skein theoretical one, which is closer to the topological construction of \cite{BHMV95}. In particular, the algebraic model is especially convenient for discussing the contravariant TQFT $\rmV'_\calC$, while the skein one is best suited for the covariant TQFT $\rmV_\calC$. Consequently, we will use the functor $\rmV'_\calC$ in the comparison with \cite{L94}, while we will focus on the functor $\rmV_\calC$ to establish the infinite order of the action of Dehn twists in the quantum group example.

\begin{remark}
If we take $\calC$ to be semisimple, at first glance the TQFT $\rmV_\calC$ is different from the Reshetikhin-Turaev TQFT, as the source category $\adCob_\calC$ has fewer cobordisms. In order to match the two theories exactly, given a (possibly non admissible) decorated cobordism $\bbM$, simply embed an unknot labeled by the tensor unit $\one$. For admissible $\bbM$, this does not change the value of $\rmV_\calC(\bbM)$, while for $\calC$ semisimple, the tensor unit is projective, and so $\bbM$ together with this additional unknot becomes admissible.
\end{remark}

In Section~\ref{sec:Hopf}, we demonstrate our TQFT construction in more explicit terms for the case when the modular category $\calC$ is of the form $\mods{H}$ for a  factorizable ribbon Hopf algebra $H$, and we show the equivalence with the construction of \cite{DGP17}, which is also based on a factorizable ribbon Hopf algebra.
 
\subsection{Projective representations of mapping class groups}

For every natural number $g$, we consider a standard closed connected surface $\Sigma_g$ of genus $g$ and a Lagrangian subspace $\lambda_g \subset H_1(\Sigma_g;\R)$. Next, for every natural number $m$ and every $m$-tuple $\underline{V} = (V_1,\ldots,V_m) \in \calC^{\times m}$, we denote with $\bbSigma_{g,\underline{V}} = (\Sigma_g,P_{\underline{V}},\lambda_g)$ the object of $\adCob_\calC$ determined by a standard set $P_{\underline{V}} \subset \Sigma_g$
of $m$ positive marked points with labels specified by the subscript. Then, we denote with $\MCG(\Sigma_g,P_{\underline{V}})$ the mapping class group of the decorated surface $(\Sigma_g,P_{\underline{V}})$, which is the group of isotopy classes of 
decoration-preserving positive self-diffeomorphisms. 

By applying the TQFT functor $\rmV_\calC$ to mapping cylinders over $\bbSigma_{g,\underline{V}}$, one obtains a map $\repT : \MCG(\Sigma_g,P_{\underline{V}}) \to \GL_\Bbbk(\rmV_\calC(\bbSigma))$.
Since the gluing operation affects the signature defect, this is in general only a projective representation, that is, in general only the induced map 
\[
 \prepT : \MCG(\Sigma_g,P_{\underline{V}}) \to \PGL_\Bbbk(\rmV_\calC(\bbSigma))
\]
is a group homomorphism. Analogously, the TQFT $\rmV'_\calC$ defines a group homomorphism
\begin{equation}\label{eq:intro-sigma'}
  \prepT' : \MCG(\Sigma_g,P_{\underline{V}}) \to \PGL_\Bbbk(\rmV'_\calC(\bbSigma)).
\end{equation}
Here, the contravariance of $\rmV'_\calC$ is compensated by choosing the mapping cylinder for the inverse element of the mapping class group
    (see Section~\ref{sec:mapping_cyl}).

The state space $\rmV'_\calC(\bbSigma_{g,\underline{V}})$ is isomorphic to a morphism space in $\calC$, namely (see Section~\ref{sss:skein-alg-summary}),
\begin{equation}
    \label{eq:intro-iso-to-alg}
 \rmV'_\calC(\bbSigma_{g,\underline{V}}) \cong \calC(V_1 \otimes \ldots \otimes V_m \otimes \coend^{\otimes g},\one),
\end{equation}
where $\coend \in \calC$ is the coend
\[
\coend  :=  \int^{X \in \calC} X^* \otimes X.
\]
By combining \eqref{eq:intro-sigma'} with \eqref{eq:intro-iso-to-alg}, we obtain the group homomorphism
\[
 \prepCpr : \MCG(\Sigma_g,P_{\underline{V}}) \to 
 \PGL_\Bbbk
 (\calC(V_1 \otimes \ldots \otimes V_m \otimes \coend^{\otimes g},\one)).
\]
Our first main result consists in computing $\prepCpr$ on a set of generators of $\MCG(\Sigma_g,P_{\underline{V}})$ in terms of the algebraic data of $\calC$ (see Proposition~\ref{P:equivalence}).

On the other hand, in \cite{L94} Lyubashenko constructs projective representations of 
 $\MCG(\Sigma_g,P_{\underline{V}})$ on the morphism space
\[
 \calC(V_1 \otimes \ldots \otimes V_m,\coend^{\otimes g}).
\]
    This is done by explicitly giving the action of a set of generators in terms of the ribbon structure of $\calC$ and of canonical morphisms associated to $\coend$, and then verifying that these satisfy the required relations.
We denote the corresponding group homomorphism by
\[
 \prepL : \MCG(\Sigma_g,P_{\underline{V}}) \to \PGL_\Bbbk
(\calC(V_1 \otimes \ldots \otimes V_m,\coend^{\otimes g})).
\]
Our second main result is (see Theorem~\ref{T:main-equiv}):

\begin{theorem}\label{T:main-intro}
The projective representations $\prepCpr$ and $\prepL$ of $\MCG(\Sigma_g,P_{\underline{V}})$ are equivalent.
\end{theorem}

\begin{remark} ~
 \begin{enumerate}
    \item Theorem \ref{T:main-intro} establishes $\rmV'_\calC$ (or, equivalently, $\rmV_\calC$) as a TQFT extension of Lyubashenko's mapping class group representations. To our knowledge, this is the first such extension in the literature (with TQFTs understood as symmetric monoidal functors with respect to disjoint union). In the approach taken here, functoriality of $\rmV'_\calC$ guarantees that we obtain a projective representation, and therefore the action of the generators we compute in Proposition~\ref{P:equivalence} automatically satisfies all
    defining relations of the mapping class group. In this sense, Theorem \ref{T:main-intro} is an independent proof of the fact that the action of the generators given in \cite{L94} indeed defines a projective representation.
    \item The renormalized Lyubashenko invariants (see Sec.~\ref{sec:renorm-Lyu}) underlying our TQFT are profoundly different from the original invariants of~\cite{L94} when the modular category $\calC$ is non-semisimple. Indeed, for admissible closed 3-manifolds, the latter are identically zero, while the former are non-trivial. Furthermore, renormalized invariants are not defined for closed 3-manifolds with empty decorations, while the original ones are well-defined, but vanish against $\Sigma \times S^1$ for every surface $\Sigma$. This is why the original Lyubashenko invariants could not be consistently extended to a TQFT.
 \end{enumerate}
\end{remark}

Finally, we consider the small quantum group $\bar{U}_q \sl_2$ at an odd root of unity $q$ as an example. For these values of $q$, the Hopf algebra $\bar{U}_q \sl_2$ can be equipped with a factorizable ribbon structure (recalled in Section~\ref{SS:small-quantum}), so that
$\calC = \mods{\bar{U}_q \sl_2}$ is a non-semisimple modular category. Let $\Sigma$ be a connected surface, let $\lambda \subset H_1(\Sigma;\R)$
be a Lagrangian subspace, and consider the object $\bbSigma = (\Sigma,\varnothing,\lambda)$ of $\adCob_\calC$, i.e. the case with no marked points. It turns out that even in this simple situation, the action of Dehn twists $\tau_\gamma$ has infinite order for \textit{every} simple closed curve $\gamma$. To establish this result, it is more convenient to use the TQFT $\rmV_\calC$ rather than $\rmV'_\calC$.

\begin{proposition}\label{P:infinite_order_Dehn_twists_intro}
 If $\calC = \mods{\bar{U}_q \sl_2}$, $\bbSigma = (\Sigma,\varnothing,\lambda)$ is a connected object of $\adCob_\calC$, and $\gamma \subset \Sigma$ is an essential simple closed curve, then $\prepT(\tau_\gamma)$ has infinite order in $\PGL_\C(\rmV_\calC(\bbSigma))$.
\end{proposition}

The proof is given in Section~\ref{SS:infinite-order}.

\subsection*{Acknowledgements}

MD was supported by KAKENHI Grant-in-Aid for JSPS Fellows 19F19765.
AMG is supported by CNRS, and partially by ANR grant JCJC ANR-18-CE40-0001  and the RSF Grant No.\ 20-61-46005. AMG is also grateful to Utah State University for its kind hospitality during February-April 2020.
NG was partially supported by the NSF grant DMS-1452093. 
IR is partially supported by the Cluster of Excellence EXC 2121.

\subsection*{Conventions}

Throughout this paper, $\Bbbk$ is an algebraically closed field, possibly of finite characteristic, and $\calC$ stands for a strict modular category over $\Bbbk$ (after the notion has been recalled in Section~\ref{subS:alg-ing}). Also, every manifold considered is oriented, and every diffeomorphism orientation-preserving.

\section{TQFTs from non-semisimple modular categories}\label{S:TQFT_manual}

In this section we review the construction of a 3-dimensional TQFT from a possibly non-semisimple modular category, following \cite{DGGPR19}. We start by recalling the necessary algebraic ingredients, then we describe how to use a modular category to define an invariant of closed 3-manifolds decorated with bichrome graphs, and next we explain how this invariant gives rise to a TQFT via the universal construction. Finally, we explain how the construction is applied when the modular category is of the form $\mods{H}$ for a finite-dimensional factorizable ribbon Hopf algebra $H$. This will be used in Appendix~\ref{A:equivalence} to show the equivalence with the construction of \cite{DGP17}.

\subsection{Algebraic ingredients}\label{subS:alg-ing}

\subsubsection{Modular tensor categories}\label{sss:MTC}

A \textit{modular tensor category}, or \textit{modular category} for short, is a finite ribbon category over $\Bbbk$ whose transparent objects are all isomorphic to direct sums of copies of the tensor unit. 

Let us unpack this definition a little bit.
A \textit{finite category} is a linear category $\calC$ over $\Bbbk$ which is equivalent to the category $\mods{A}$ of finite dimensional left $A$-modules for some finite dimensional algebra $A$ over $\Bbbk$. A \textit{ribbon structure} on a linear category $\calC$ is given by:
\begin{enumerate}
 \item a tensor product $\otimes : \calC \times \calC \to \calC$, together with associativity isomorphisms;
 \item a tensor unit $\one \in \calC$, together with unit isomorphisms;
 \item duality morphisms $\lev_X : X^* \otimes X \to \one$, $\lcoev_X : \one \to X \otimes X^*$, and $\rev_X : X \otimes X^* \to \one$, $\rcoev_X : \one \to X^* \otimes X$ for every $X \in \calC$;
 \item braiding isomorphisms $c_{X,Y} : X \otimes Y \to Y \otimes X$ for all $X,Y \in \calC$;
 \item twist isomorphisms $\theta_X : X \to X$ for every $X \in \calC$.
\end{enumerate}
All these pieces of structure are subject to a number of conditions that can be found in \cite{EGNO15}. For notational simplicity we will assume $\calC$ to be strict, that is, associativity and unit isomorphisms are identities.

An object $X \in \calC$ of a ribbon category is \textit{transparent} if its double braiding with any other object of $\calC$ is trivial, meaning $c_{Y,X} \circ c_{X,Y} = \id_{X \otimes Y}$ for every $Y \in \calC$.

\subsubsection{The coend $\coend$}\label{sec:def_coend}

A modular category $\calC$ admits a \textit{coend}
\[
 \coend := \int^{X \in \calC} X^* \otimes X \in \calC,
\]
which is defined as the universal dinatural transformation with source
\begin{align*}
  (\_^* \otimes \_) : \calC^\op \times \calC & \to \calC \\*
  (X,Y) & \mapsto X^* \otimes Y,
\end{align*}
see \cite{M71} for details. In particular, $\coend$ is an object of $\calC$ equipped with structure morphisms $i_X : X^* \otimes X \to \coend$ for every $X \in \calC$ which are dinatural in $X$. The coend $\coend$ carries the structure of a braided Hopf algebra in $\calC$
\cite{M93, L95}, meaning it is equipped with:
\begin{enumerate}
 \item a product $\mu : \coend \otimes \coend \to \coend$ and a unit $\eta : \one \to \coend$;
 \item a coproduct $\Delta : \coend \to \coend \otimes \coend$ and a counit $\epsilon : \coend \to \one$;
 \item an antipode $S : \coend \to \coend$.
\end{enumerate}
Moreover, the coend $\coend$ admits non-zero two-sided integrals and cointegrals, which are unique up to scalar, see e.g.~\cite[Sec.~4.2.3]{KL01}. 
By definition\footnote{In general, the source of an integral and the target of a cointegral are given by an invertible object other than the tensor unit, the so-called object of integrals. For modular categories, the object of integrals is the tensor unit. See e.g.~\cite[Sec.~2]{DGGPR19} for more details and references.
}, a \textit{two-sided integral}, or \textit{integral} for short,
is a morphism $\intL : \one \to \coend$ satisfying
\[
 \mu \circ (\intL \otimes \id_\coend) = \intL \circ \epsilon =
 \mu \circ (\id_\coend \otimes \intL),
\]
and a \textit{two-sided cointegral}, or \textit{cointegral} for short,
is a morphism $\cointL : \coend \to \one$ satisfying
\[
 (\cointL \otimes \id_\coend) \circ \Delta = \eta \circ \cointL =
 (\id_\coend \otimes \cointL) \circ \Delta.
\]
Every modular category also admits an \textit{end}
\[
 \eend := \int_{X \in \calC} X \otimes X^* \in \calC,
\]
which comes equipped with a dinatural family of morphisms $j_X : \eend \to X \otimes X^*$. The universal properties of $\coend$ and $\eend$ imply that there is a unique morphism $D : \coend \to \eend$, called the \textit{Drinfeld map}, such that for all $X,Y \in \calC$
\[
    j_Y \circ D \circ i_X = \pic{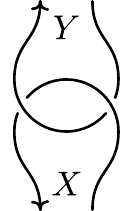}
    : \; X^* \otimes X \to Y \otimes Y^*.
\]
Here we use standard diagrammatic calculus, with diagrams read from bottom to top, see \cite{DGGPR19} for further conventions.
As a consequence of modularity, $D$ is invertible.
In fact, there are three a priori independent non-degeneracy conditions on the braiding of a finite ribbon category. First, that there are no non-trivial transparent objects, as used above in the definition of modularity; Second, that the Drinfeld map is invertible; Third, that a certain functor to the Drinfeld center of $\calC$ is an equivalence. These conditions were shown to be equivalent in \cite{S16}, and hence all three can be used to define modularity. Invertibility of $D$ was the condition originally used in \cite{L94}.\footnote{
    Actually, in \cite{L94} the non-degeneracy of a certain Hopf pairing on $\coend$ is used as the defining condition, but this is easily seen to be equivalent to the invertibility of $D$ \cite[Prop.~4.11]{FGR17}.} 
However, the condition on transparent objects is the easiest one to state.

\subsubsection{Modified traces}\label{sss:mod-tr-summary}

Denote by $\Proj(\calC)$ the full subcategory of projective objects of $\calC$. Then $\Proj(\calC)$ is a tensor ideal in $\calC$, that is, it is absorbent with respect to tensor products with arbitrary objects of $\calC$, and it is closed with respect to retracts. In fact, $\Proj(\calC)$ is the smallest non-zero tensor ideal of $\calC$ \cite[Sec.~4.4]{GKP10}.

A \textit{modified trace} on $\Proj(\calC)$ is a family of linear maps
\[
 \rmt := \{ \rmt_X : \End_\calC(X) \to \Bbbk \mid X \in \Proj(\calC) \}
\]
satisfying:
\begin{enumerate}
 \item \textit{Cyclicity}: $\rmt_X(g \circ f) = \rmt_Y(f \circ g)$ for all objects $X,Y \in \Proj(\calC)$, and all morphisms $f \in \calC(X,Y)$ and $g \in \calC(Y,X)$;
 \item \textit{Partial trace}: $\rmt_{X \otimes Y}(f) = \rmt_X(\ptr(f))$ for all objects $X \in \Proj(\calC)$ and $Y \in \calC$, and every morphism $f \in \End_\calC(X \otimes Y)$, where the morphism $\ptr(f) \in \End_\calC(X)$ is defined as
  \[
   \ptr(f) := (\id_X \otimes \rev_Y) \circ (f \otimes \id_{Y^*}) \circ (\id_X \otimes \lcoev_Y).
  \]
\end{enumerate}

One can show that, since $\calC$ is modular, there exists a non-zero modified trace on $\Proj(\calC)$, and that this trace is unique up to an overall scalar \cite{GKP11,GR17}, see also \cite[Cor.~5.6]{GKP18} for a much more general existence result. Moreover, this modified trace induces a non-degenerate pairing $\calC(X,Y) \times \calC(Y,X) \to \Bbbk$ given by $(f,g) \mapsto t_Y(f \circ g)$ for all $X \in \calC$ and $Y \in \Proj(\calC)$ \cite{GR17,GKP18}.
By contrast, the standard categorical trace vanishes on $\Proj(\calC)$ as soon as $\calC$ is non-semisimple.

Since $\calC$ is finite, it has enough projectives and injectives. In particular, the tensor unit $\one$ has a projective cover $P_{\one} \in \Proj(\calC)$, equipped with a surjection morphism $\epsilon_{\one} : P_{\one} \to \one$. Since $\calC$ is modular, and hence also unimodular, $P_{\one}$ is in addition the injective envelope of $\one$, see e.g.~\cite[Sec.~2]{DGGPR19} for more details. In particular, there is an injection morphism $\eta_{\one} : \one \to P_{\one}$. Note that $\epsilon_{\one} \circ \eta_{\one} \neq 0$ if and only if $\calC$ is semisimple. Both $\epsilon_{\one}$ and $\eta_{\one}$ are unique up to scalars, and we link their normalization to that of the modified trace via the condition
\[
 \rmt_{P_{\one}}(\eta_{\one} \circ \epsilon_{\one}) = 1.
\]
The fact that this expression is non-zero to begin with follows from the non-degeneracy of the modified trace.

\subsection{3-Manifold invariants}

\subsubsection{Lyubashenko-Reshetikhin-Turaev functor}\label{sss:LRT-functor}

Fix a non-zero integral $\intL$ for $\coend$.
Here we recall an extension $F_\intL$ of the Reshetikhin-Turaev functor $F_\calC$ of \cite{T94}. While the source of $F_\calC$ is the category $\calR_\calC$ of $\calC$-colored ribbon graphs, $F_\intL$ is defined over the category $\calR_\intL$ of so-called bichrome graphs \cite{DGP17,DGGPR19}.

Bichrome graphs are a generalization of standard ribbon graphs whose edges can be of two kinds, \textit{red} (carrying no labels), and \textit{blue} (labeled as usual by objects of $\calC$). For what concerns coupons, they can be, according to the edges intersecting them, either \textit{bichrome} (and unlabeled), or \textit{blue} (and labeled as usual by morphisms of $\calC$), while red coupons are forbidden. Furthermore, bichrome coupons only have two possible configurations, namely (color available online)
\[
 \includegraphics{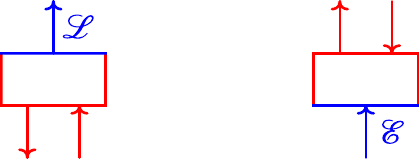}
\]
with blue edges labeled by the coend $\coend$ and the end $\eend$, respectively (see Section~\ref{sec:def_coend}). 
    Blue coupons, on the other hand,
can be arbitrary. Boundary vertices can only meet blue edges. 

To evaluate the functor $F_\intL : \calR_\intL \to \calC$ on morphisms, i.e.\ on bichrome graphs, one carries out several steps, which we illustrate in the following example:
\[
    T = \pic{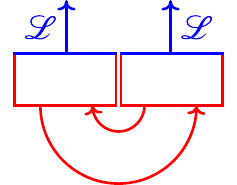}
\]
Remark that throwing away the blue part of the graph, as well as bichrome coupons, and joining together each pair of red strands which did intersect a bichrome coupon, results in a red link.
Then, start by cutting a red edge of the bichrome graph for each component of the red link, and drag the ends to the bottom, placing them side by side, with the downward-oriented strand on the left, and the upward-oriented one on the right. For $n$ red components, this results in a so-called $n$-bottom graph, with $2n$ red boundary vertices at the bottom. In our example, $n=1$ and 
the result of the above operation is
\[
    \tilde T = \pic{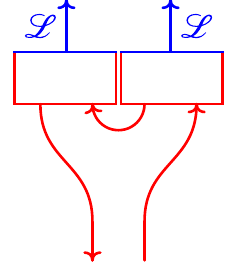}
\]
The plat closure operation, which consists in joining together in pairs red strands intersecting adjacent incoming boundary vertices, has to result in the original bichrome graph.
Next, choose objects $X_1, \dots, X_n$ (one for each of the $n$ components of the red link), label red edges correspondingly by $X_k$, for $1\leq k \leq n$, and label bichrome coupons meeting these edges by the dinatural morphism $i_{X_k}$, in the case of $\coend$, or $j_{X_k}$, in the case of $\eend$. In our example one obtains the graph
\[
    \tilde T_{X_1} = \pic{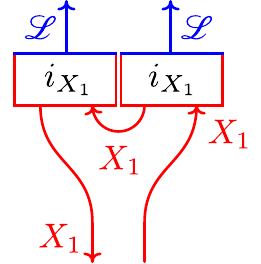}
\]
At this point, all edges and coupons are labeled, and forgetting the difference between red and blue we have a $\calC$-colored ribbon graph, that is, a morphism in $\calR_\calC$. To this we can apply the Reshetikhin-Turaev functor $F_\calC$ to obtain a morphism in~$\calC$. In the above example, 
\[
    F_\calC( \tilde T_{X_1} ) = \big( i_{X_1} \otimes i_{X_1} \big) \circ \big( \id_{X_1^*} \otimes \lcoev_{X_1} \otimes \id_{X_1} \big).
\]
By construction, the resulting family of morphisms (indexed by $X_1,\dots,X_n$) is dinatural, and via the universal property of $\coend$ defines a morphism out of $\coend^{\otimes n}$, possibly tensored with other objects coming from blue boundary vertices. For our bichrome graph $T$ we obtain a morphism $f_\calC(T) : \coend \to \coend \otimes \coend$, which is uniquely specified by the condition\footnote{We compute these  morphisms explicitly for general  $n$-bottom graphs in Lemma~\ref{L:comparison_H-L} in the case $\calC = \mods{H}$ for a finite-dimensional ribbon Hopf algebra $H$.}
\[
    f_\calC(T) \circ i_{X_1} = F_\calC( \tilde T_{X_1} ).
\]
This is the defining property of the coproduct of $\coend$, so that in fact $f_\calC(T) = \Delta$. The final step is to pre-compose with the $n$-fold tensor power of the integral $\intL$, so that in the example,  $F_\intL( T ) : \one \to \coend \otimes \coend$ is given by
\[
    F_\intL( T ) = \Delta \circ \intL.
\]
Note that while $F_\intL$ does depend on the choice of $\intL$, the category $\calR_\intL$ does not. The notation is chosen merely to remind of the relevant functor.
We refer to \cite[Sec.~3.1]{DGGPR19} for more details on bichrome graphs, the category $\calR_\intL$, and the definition (and well-definedness) of $F_\intL$. 

By construction, $\calR_\calC$ and $\calR_\intL$ have the same objects, and there is a commuting (on the nose) diagram of functors
\begin{center}
  \begin{tikzpicture}[descr/.style={fill=white}]
  \node (P1) at (165:2.25) {$\calR_\calC$};
  \node (P2) at (195:2.25) {$\calR_\intL$};
  \node (P3) at (0:0) {$\calC \hspace*{10pt}$};
  \draw
  (P1) edge[right hook->] (P2)
  (P1) edge[->] node[above] {\scriptsize $F_\calC$} (P3)
  (P2) edge[->] node[below] {\scriptsize $F_\intL$} (P3);
 \end{tikzpicture}
\end{center}
where the vertical arrow is the identity at the level of objects and the inclusion of purely blue ribbon graphs into bichrome graphs at the level of morphisms.

\subsubsection{Renormalized Lyubashenko invariant}\label{sec:renorm-Lyu}

We say a bichrome graph is \textit{closed} if it has no boundary vertices, and we say it is \textit{admissible} if it has at least one blue edge labeled by a projective object of $\calC$.
Every admissible closed bichrome graph $T$ admits a cutting presentation, which is a bichrome graph $T_V$ featuring a single incoming boundary vertex and a single outgoing one, both positive and labeled by $V \in \Proj(\calC)$, and whose trace closure is $T$. If $T$ is a closed admissible bichrome graph and $T_V$ is a cutting presentation, then the scalar
\[
 F'_\intL(T) := \rmt_V(F_\intL(T_V))
\]
is independent of $T_V$ and is a topological invariant of $T$ \cite[Thm.~3.3]{DGGPR19}.

By suitably normalizing it, $F'_\intL$ can be used to define an invariant $\rmL'_\calC$ of admissible decorated 3-manifolds. A \textit{decorated 3-manifold} is a pair $(M,T)$, where $M$ is a connected closed 3-manifold, and where $T \subset M$ is a closed bichrome graph. A pair $(M,T)$ is \textit{admissible} if $T$ is.

To define the invariant, we need \textit{stabilization coefficients}, which are defined as
\[
 \Delta_\pm := F_\intL(O_\pm) \in \Bbbk,
\]
where $O_\pm$ denotes a red $\pm 1$-framed unknot.
These coefficients are automatically non-zero \cite{KL01} (see \cite[Prop.~2.6]{DGGPR19} for a proof in the present notation), and thus allow us to fix constants $\calD$ and $\delta$ given by
\begin{equation}\label{eq:delta-def}
 \calD^2 = \Delta_+ \Delta_-, \qquad 
 \delta = \frac{\calD}{\Delta_-} = \frac{\Delta_+}{\calD}.
\end{equation}
Note that $\calD$ involves a choice of square root.

Let $(M,T)$ be an admissible decorated 3-manifold, and let $L$ be a surgery presentation of $M$ given by a red framed oriented link in $S^3$ with $\ell$ components and signature $\sigma$. We assume the bichrome graph $T$ is contained in the exterior of the surgery link $L$, so that we can think of them as simultaneously embedded in $S^3$.
Then, as proved in \cite[Thm.~3.8]{DGGPR19}, the scalar
\begin{equation}\label{eq:L'-via-F'}
 \rmL'_\calC(M,T) := \calD^{-1-\ell} \delta^{-\sigma} F'_\intL(L \cup T)
\end{equation}
is a topological invariant of the pair $(M,T)$.
We will call $\rmL'_\calC(M,T)$ the \textit{renormalized Lyubashenko invariant} of the admissible decorated 3-manifold $(M,T)$. The use of modified traces to define renormalized 3-manifold invariants via link surgery was pioneered in \cite{CGP12}.

\newpage

\subsection{Construction of TQFTs}\label{SS:TQFT_data}

\subsubsection{Admissible cobordisms}\label{sss:admissible_cobord}

The TQFT associated to $\calC$ will be defined on a cobordism category that contains the same objects but fewer morphisms than the one used for the original Reshetikhin-Turaev TQFT \cite{T94}.

The \textit{category $\adCob_\calC$ of admissible cobordisms} has objects and morphisms defined as follows.
An object $\bbSigma$ of $\adCob_{\calC}$ is a triple $(\Sigma,P,\lambda)$ where: 
\begin{enumerate}
 \item $\Sigma$ is a closed surface;
 \item $P \subset \Sigma$ is a finite set of oriented framed points labeled by objects of $\calC$;
 \item $\lambda \subset H_1(\Sigma;\R)$ is a Lagrangian subspace for the intersection form.
\end{enumerate}
A morphism $\bbM : \bbSigma \rightarrow \bbSigma'$ of $\adCob_{\calC}$ is an equivalence class of admissible triples $(M,T,n)$ where:
\begin{enumerate}
 \item $M$ is a 3-dimensional cobordism from $\Sigma$ to $\Sigma'$;
 \item $T \subset M$ is a bichrome graph from $P$ to $P'$;
 \item $n \in \Z$ is an integer called the \textit{signature defect}.
\end{enumerate}
A triple $(M,T,n)$ is \textit{admissible} if every connected component of $M$ disjoint from the incoming boundary $\partial_- M \cong \Sigma$ contains an admissible subgraph of $T$.
Two triples $(M,T,n)$ and $(M',T',n')$ are equivalent if $n = n'$ and if there exists an isomorphism of cobordisms $f : M \rightarrow M'$ satisfying $f(T) = T'$.  

The composition $\bbM' \circ \bbM : \bbSigma \rightarrow \bbSigma''$ of morphisms $\bbM' : \bbSigma' \rightarrow \bbSigma''$, $\bbM : \bbSigma \rightarrow \bbSigma'$  is the equivalence class of the triple
\begin{equation}\label{eq:compose-Cob}
 \Big( M \cup_{\Sigma'} M', T \cup_{P'} T', n + n' - \mu \big( M_*(\lambda),\lambda',M'^* (\lambda'') \big) \Big)
\end{equation}
for the Lagrangian subspaces 
\begin{gather*}
 M_*(\lambda) := \{ x' \in H_1(\Sigma';\R) \mid (i_{M_+})_*(x') \in (i_{M_-})_* (\lambda) \}, \\*
 M'^*(\lambda'') := \{ x' \in H_1(\Sigma';\R) \mid (i_{M'_-})_*(x') \in (i_{M'_+})_*(\lambda'') \}.
\end{gather*}
Here,
\[
 i_{M_-} : \Sigma \hookrightarrow M, \quad 
 i_{M_+} : \Sigma' \hookrightarrow M, \quad
 i_{M'_-} : \Sigma' \hookrightarrow M', \quad 
 i_{M'_+} : \Sigma'' \hookrightarrow M'
\]
are the embeddings induced by the structure maps of $M$ and $M'$, and $\mu$ denotes the Maslov index, see \cite{T94} for a detailed account of its properties.

One notable difference between $\adCob_\calC$ and the full cobordism category is that in $\adCob_\calC$ not every object is dualizable. Indeed, the dualizable objects $\bbSigma=(\Sigma,P,\lambda)$ in $\adCob_\calC$ are precisely those where each connected component of $\Sigma$ contains at least one vertex of $P$ labeled by a projective object of $\calC$. If this condition is not satisfied, the decorated cobordism giving the coevaluation\footnote{That is, the cylinder over $\bbSigma$ with both boundary components declared as outgoing.} in the full cobordism category is not admissible, and so it is not in $\adCob_\calC$.

\subsubsection{TQFT via the universal construction}\label{sss:TQFT-from-universal}

We can use the universal construction of \cite{BHMV95} to extend an invariant of closed manifolds to a functor defined on cobordisms.
To start with, we generalize $\rmL'_\calC$ to an invariant of closed morphisms of $\adCob_\calC$ by setting
\begin{equation}
\label{eq:L'-on-cobordisms}
 \rmL'_\calC(\bbM) := \delta^n \rmL'_\calC(M,T)
\end{equation}
for every closed connected morphism $\bbM = (M,T,n)$, and by setting
\[
 \rmL'_\calC(\bbM_1 \disjun \ldots \disjun \bbM_k) := \prod_{i=1}^k \rmL'_\calC(\bbM_i)
\]
for every tensor product of closed connected morphisms $\bbM_1, \ldots, \bbM_k$. Then, if $\bbSigma$ is an object of $\adCob_\calC$, we consider vector spaces $\calV(\bbSigma)$ and $\calV'(\bbSigma)$ with basis given by morphisms of $\adCob_\calC$ of the form $\bbM_{\bbSigma} : \varnothing \to \bbSigma$ and $\bbM'_{\bbSigma} : \bbSigma \to \varnothing$, respectively. These vector spaces are paired via
\begin{align*}
 \langle \_,\_ \rangle_{\bbSigma} : \calV'(\bbSigma) \times \calV(\bbSigma) & \to \Bbbk \\*
 (\bbM'_{\bbSigma},\bbM_{\bbSigma}) & \mapsto \rmL'_\calC(\bbM'_{\bbSigma} \circ \bbM_{\bbSigma}).
\end{align*}
State spaces are then defined as quotients with respect to the left and the right radical of the pairing $\langle \_,\_ \rangle_{\bbSigma}$, namely
\[
 \rmV_\calC(\bbSigma) := \calV(\bbSigma) / \rrad \langle \_,\_ \rangle_{\bbSigma}, \qquad
 \rmV'_\calC(\bbSigma) := \calV'(\bbSigma) / \lrad \langle \_,\_ \rangle_{\bbSigma}.
\]
    We will denote equivalence classes in these quotients by square brackets.
Now, if $\bbM : \bbSigma \to \bbSigma'$ is a morphism of $\adCob_\calC$, its associated operators are simply defined as
\begin{align*}
 \rmV_\calC(\bbM) : \rmV_\calC(\bbSigma) & \to \rmV_\calC(\bbSigma') 
 &
 \rmV'_\calC(\bbM) : \rmV'_\calC(\bbSigma') & \to \rmV'_\calC(\bbSigma) \\*
 {}[\bbM_{\bbSigma}] & \mapsto [\bbM \circ \bbM_{\bbSigma}], &
 {}[\bbM'_{\bbSigma}] & \mapsto [\bbM'_{\bbSigma} \circ \bbM].
\end{align*}
Altogether, this defines a pair of functors
\[
 \rmV_\calC : \adCob_\calC \to \Vect_\Bbbk, \quad 
 \rmV'_\calC : (\adCob_\calC)^{\op} \to \Vect_\Bbbk.
\]
These functors are dual to each other in the sense that the vector spaces
$\rmV_\calC(\bbSigma)$ and $\rmV'_\calC(\bbSigma)$ are non-degenerately paired 
via $\langle \_,\_ \rangle_{\bbSigma}$, and that for all $\bbM : \bbSigma \to \bbSigma'$, $[\bbM_{\bbSigma}] \in \rmV_\calC(\bbSigma)$, and $[\bbM'_{\bbSigma'}] \in \rmV'_\calC(\bbSigma')$ we have
\begin{equation}\label{eq:V-V'-dual}
\big\langle \bbM'_{\bbSigma'}, \rmV_\calC(\bbM)(\bbM_{\bbSigma}) 
\big\rangle_{\bbSigma'} =
\big\langle \rmV'_\calC(\bbM)(\bbM'_{\bbSigma'}), \bbM_{\bbSigma} \big\rangle_{\bbSigma}
~.
\end{equation}
It is shown in \cite[Thm.~4.12]{DGGPR19} that $\rmV_\calC$ and $\rmV'_\calC$ are symmetric monoidal, so that they define TQFTs on $\adCob_\calC$ and $(\adCob_\calC)^{\op}$, respectively.

\subsubsection{Skein modules and algebraic state spaces}\label{sss:skein-alg-summary}

While the definition of the state spaces $\rmV_\calC(\bbSigma)$ and $\rmV'_\calC(\bbSigma)$ in terms of equivalence classes of linear combinations of decorated cobordisms is very convenient for the construction of the TQFT functors, it is not at all obvious how to express these spaces in terms of the underlying modular category $\calC$. We will do this in two steps, first introducing admissible skein modules, and then passing to (quotients of) certain morphism spaces in $\calC$.

If $\bbSigma = (\Sigma,P,\lambda)$ is an object of $\adCob_\calC$ and $M$ is a connected 3-dimensional cobordism from $\varnothing$ to $\Sigma$, then we denote with $\calV(M;\bbSigma)$ the vector space generated by isotopy classes of admissible bichrome graphs $T$ inside $M$ from $\varnothing$ to $P$. Similarly, if $\bbSigma$ is non-empty and $M'$ is a connected 3-dimensional cobordism from $\Sigma$ to $\varnothing$, then we denote with $\calV'(M';\bbSigma)$ the vector space generated by isotopy classes of
(not necessarily admissible)
bichrome graphs $T'$ inside $M'$ from $P$ to $\varnothing$. 

We say two vectors of $\calV'(M;\bbSigma)$ are \textit{skein equivalent} if, up to isotopy, they are related by a finite sequence of local moves replacing
parts of bichrome graphs obtained by restriction to topological 3-balls with bichrome graphs having the same image under the functor $F_\intL$. 

On the other hand, for $\calV(M;\bbSigma)$ we introduce an extra condition for skein equivalence. Namely, we say a local move is \textit{admissible} if the complement of the topological 3-ball inside which it takes place contains an admissible subgraph. We say two vectors of $\calV(M;\bbSigma)$ are \textit{admissibly skein equivalent}, or \textit{skein equivalent} for short, if they are related by a finite sequence of admissible local moves.

We define the \textit{skein modules} $\adSk(M;\bbSigma)$ and $\Sk'(M';\bbSigma)$ as the quotients of $\calV(M;\bbSigma)$ and $\calV'(M';\bbSigma)$ with respect to skein equivalence, respectively.
Then the natural linear maps
\begin{align*}
 \pi : \adSk(M;\bbSigma) & \to \rmV_\calC(\bbSigma) 
 &
 \pi' : \Sk'(M';\bbSigma) & \to \rmV'_\calC(\bbSigma) \\*
 {}[T] & \mapsto [M,T,0]
 &
 {}[T'] & \mapsto [M',T',0]
\end{align*}
are surjective \cite[Prop.~4.11]{DGGPR19}.

Next, let us recall the algebraic model for state spaces associated with connected objects of $\adCob_\calC$. For all $g,m \in \N$ and every $\underline{V} = (V_1,\ldots,V_m) \in \calC^{\times m}$, let us consider vector spaces\footnote{In \cite{DGGPR19} we used $\calC(P_{\one},\eend^{\otimes g} \otimes V_1 \otimes \ldots \otimes V_m)$ and $\calC(\coend^{\otimes g} \otimes V_1 \otimes \ldots \otimes V_m ,\one)$ instead, i.e. we placed (co)ends as first tensor factors, rather than last. The morphism spaces we use here are isomorphic (via the braiding) and turn out to be more convenient when relating to Lyubashenko's mapping class group representations.
}
\[
 \calX_{g,\underline{V}} := \calC(P_{\one},V_1 \otimes \ldots \otimes V_m \otimes \eend^{\otimes g}), \quad
 \calX'_{g,\underline{V}} := \calC(V_1 \otimes \ldots \otimes V_m \otimes \coend^{\otimes g},\one),
\]
where $P_{\one}$ is the projective cover of the tensor unit $\one$. Using the injection morphism $\eta_{\one}: \one \to P_{\one}$ from Section~\ref{sss:mod-tr-summary} and the inverse of the Drinfeld map $D$, these vector spaces are paired via
\begin{align*}
 \langle \_,\_ \rangle_{g,\underline{V}} : \calX'_{g,\underline{V}} \times \calX_{g,\underline{V}} 
 & \to \Bbbk \\*
 (x',x) 
 & \mapsto t_{P_{\one}}(\eta_{\one} \circ x' \circ (\id_{V_1 \otimes \ldots \otimes V_m} \otimes (D^{-1})^{\otimes g}) \circ x ) \\*
 & \phantom{\mapsto} \hspace*{\parindent} = \langle x' \circ \left( (\id_{V_1 \otimes \ldots \otimes V_m} \otimes (D^{-1})^{\otimes g} \right) \circ x \rangle_{\epsilon_{\one}},
\end{align*}
where for every $f \in \Hom_H(P_{\one},\one)$ the scalar $\langle f \rangle_{\epsilon_{\one}} \in \Bbbk$ is defined by $f = \langle f \rangle_{\epsilon_{\one}} \epsilon_{\one}$ for the projective cover morphism $\epsilon_{\one} : P_{\one} \to \one$.
As proved in \cite[Lem.~4.1]{DGGPR19}, the left radical $\lrad \langle \_,\_ \rangle_{g,\underline{V}}$ of the pairing $\langle \_,\_ \rangle_{g,\underline{V}}$ is trivial.
We define the \textit{algebraic state spaces} as
\begin{equation}\label{eq:alg-statespace}
 \rmX_{g,\underline{V}} := \calX_{g,\underline{V}} / \rrad \langle \_,\_ \rangle_{g,\underline{V}}, \quad
 \rmX'_{g,\underline{V}} := \calX'_{g,\underline{V}}.
\end{equation}
They are dual to each other by construction.

In the following we will use skein modules to give an isomorphism between state spaces coming from the functors $\rmV_\calC$ and $\rmV'_\calC$ and their algebraic models.
For every natural number $g$, we consider a standard closed connected surface $\Sigma_g$ of genus $g$ equipped with a fixed Lagrangian subspace $\lambda_g \subset H_1(\Sigma_g;\R)$. For every natural number $m$ and every $m$-tuple $\underline{V} = (V_1,\ldots,V_m) \in \calC^{\times m}$, we denote with 
\begin{equation}\label{eq:standard-surf_gV}
\bbSigma_{g,\underline{V}} = (\Sigma_g,P_{\underline{V}},\lambda_g)
\end{equation}
the object of $\adCob_\calC$ whose decorations are specified by the subscript. We want to describe isomorphisms
\[
 \rmV_\calC(\bbSigma_{g,\underline{V}}) \cong \rmX_{g,\underline{V}},
 \quad
 \rmV'_\calC(\bbSigma_{g,\underline{V}}) \cong \rmX'_{g,\underline{V}}.
\]
We start by fixing a genus $g$ Heegaard splitting $M_g \cup_{\Sigma_g} M'_g$ of $S^3$, which we represent as in Figure~\ref{F:algebraic_model}. 
With respect to this configuration, we define the map
\[
 \Psi : \rmX_{g,\underline{V}} \to \adSk(M_g;\bbSigma_{g,\underline{V}}),
\]
which sends $[x] \in \rmX_{g,\underline{V}}$ to the admissible bichrome graph contained in the bottom handlebody in Figure~\ref{F:algebraic_model}. One can check that this assignment is independent of the choice of a representative $x \in \calX_{g,\underline{V}}$ of the class $[x]$. Denote by $\Phi$ the map obtained by composing with the projection $\pi$,
\[
    \Phi = \Big[ \rmX_{g,\underline{V}} \xrightarrow{\Psi} \adSk(M_g;\bbSigma_{g,\underline{V}})  \xrightarrow{\pi} \rmV_\calC(\bbSigma_{g,\underline{V}})  \Big].
\]
It has been shown in \cite[Prop.~4.17]{DGGPR19} that $\Phi$ is an isomorphism. Similarly, one defines the map 
\[
 \Psi' : \rmX'_{g,\underline{V}} \to \Sk'(M'_g;\bbSigma_{g,\underline{V}}),
\]
which sends $x' \in \rmX'_{g,\underline{V}}$ to the bichrome graph contained in the top handlebody in Figure~\ref{F:algebraic_model}. The composition with $\pi'$ results in an isomorphism
\begin{equation}\label{eq:Phi'-def}
    \Phi' = \Big[ \rmX'_{g,\underline{V}} \xrightarrow{\Psi'}
    \Sk'(M'_g;\bbSigma_{g,\underline{V}})  
    \xrightarrow{\pi'} \rmV'_\calC(\bbSigma_{g,\underline{V}})  \Big].
\end{equation}

Let $\bbM : \bbSigma_{g,\underline{V}} \to \bbSigma_{g',\underline{V'}}$ be a morphism of $\adCob_\calC$ of the form $\bbM = (M,T,n)$. In order to compute its action on $x \in \rmX_{g,\underline{V}}$, we need to find the unique $y \in \rmX_{g',\underline{V'}}$ satisfying
\[
    \Phi(y) = \rmV_\calC(\bbM)\big( \Phi(x) \big) \in \rmV_\calC(\bbSigma_{g',\underline{V'}}).
\]
Below we will be interested in projective actions, which allow us to ignore the signature defect. Up to a scalar, $y$ is then determined by the proportionality relation
\[
 [M_{g'},\Psi(y),0] \propto [M_g \cup_{\Sigma_g} M,\Psi(x) \cup_{P_{\underline{V}}} T,0].
\]
Similarly, in order to compute the right projective action of $\bbM$ on $x' \in \rmX'_{g',\underline{V'}}$, we need to find some $y' \in \rmX'_{g,\underline{V}}$ satisfying
\[
 [M'_g,\Psi'(y'),0] \propto [M \cup_{\Sigma_{g'}} M'_{g'},T \cup_{P_{\underline{V'}}} \Psi'(x'),0].
\]

\begin{figure}[t]
 \includegraphics{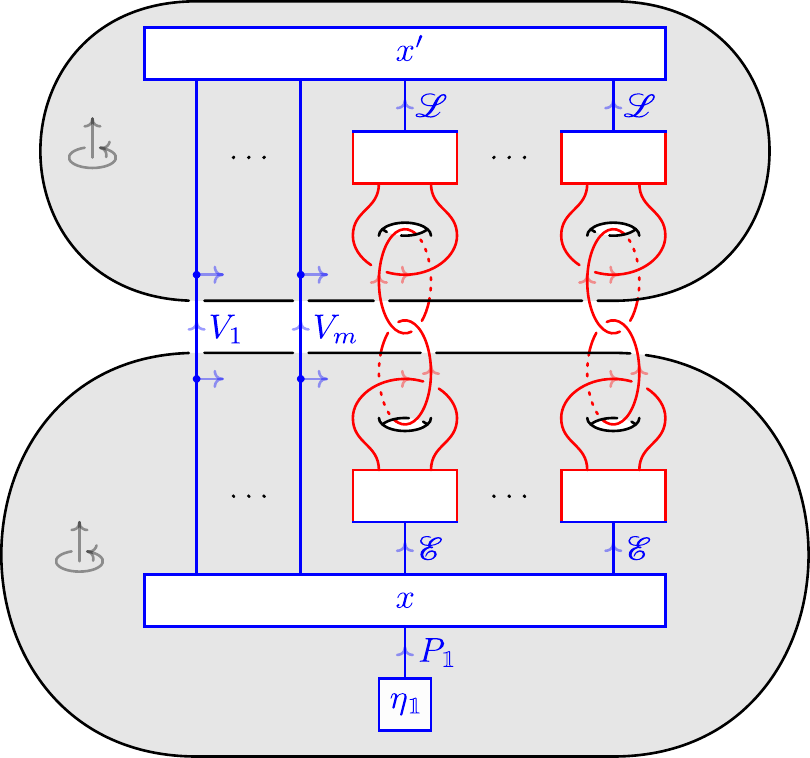}
 \caption{Isomorphism between algebraic and skein models.}
 \label{F:algebraic_model}
\end{figure} 

\subsection{Hopf algebra case}\label{sec:Hopf}

\subsubsection{Factorizable finite-dimensional ribbon Hopf algebras}\label{subsec:Hopf}

Let us translate all the main ingredients of the construction in Hopf algebraic terms when $\calC = \mods{H}$ for a finite-di\-men\-sion\-al ribbon Hopf algebra $H$ over $\Bbbk$. We recall that this means $H$ comes equipped with:
\begin{enumerate}
 \item a \textit{product} $\mu : H \otimes H \to H$ and a \textit{unit} $\eta : \one \to H$;
 \item a \textit{coproduct} $\Delta : H \to H \otimes H$ and a \textit{counit} $\epsilon : H \to \one$;
 \item an \textit{antipode} $S : H \to H$;
 \item an \textit{R-matrix} $R \in H \otimes H$ and a \textit{ribbon element} $v \in H$.
\end{enumerate}
These data should satisfy several axioms, see \cite[Def.~VII.2.2 \& XIV.6.1]{K95}. We adopt some standard notational conventions, like the short notations for product and unit
\[
 \mu(x \otimes y) = xy, \qquad \eta(1) = 1
\]
for all $x,y \in H$, and Sweedler's notation for the coproduct
\[
 \Delta(x) = x_{(1)} \otimes x_{(2)}
\]
for every $x \in H$, which hides a sum. The R-matrix $R = R' \otimes R''$ determines a 
	\textit{Drinfeld element} 
$u \in H$ defined by
\[
 u := S(R'')R'.
\]
Remark that here too a sum is hidden. Then, the Drinfeld and the ribbon element determine a \textit{pivotal element} $g \in H$ defined by
\begin{equation}\label{eq:pivot}
 g := uv^{-1}.
\end{equation}

The finite category $\calC = \mods{H}$ of finite-dimensional left $H$-modules supports the structure of a ribbon category. The monoidal structure is induced by the coproduct $\Delta$ and the counit $\epsilon$ as explained in \cite[Sec.~XI.3.1]{K95}, and the left rigid structure is induced by the antipode $S$ as explained in Example~1 of \cite[Sec.~XIV.2]{K95}. Furthermore:
\begin{enumerate}
 \item The pivotal element $g$ defines, for every $V \in \calC$, the right duality morphisms $\rev_V : V \otimes V^* \to \Bbbk$ and $\rcoev_V : \Bbbk \to V^* \otimes V$ determined by
  \[
   \rev_V(v \otimes \phi) := \phi(g \cdot v), \qquad
   \rcoev_V(1) := \sum_{i=1}^n \phi^i \otimes (g^{-1} \cdot v_i)
  \]
  for all $v \in V$, $\phi \in V^*$, where $\{ v_i \in V \mid 1 \leq i \leq n \}$, $\{ \phi^i \in V^* \mid 1 \leq i \leq n \}$ are dual bases;
 \item The R-matrix $R$ defines, for all $V,W \in \calC$, the braiding isomorphism $c_{V,W} : V \otimes W \to W \otimes V$ determined by
  \[
   c_{V,W}(v \otimes w) := (R'' \cdot w) \otimes (R' \cdot v)
  \]
  for all $v \in V$, $w \in W$;
 \item The ribbon element $v$ defines, for every $W \in \calC$, the twist isomorphism $\theta_W : W \to W$ determined by
  \[
   \theta_W(w) := v^{-1} \cdot w
  \]
  for every $w \in W$.
\end{enumerate}

The coend $\coend = \int^{V \in \calC} V^* \otimes V$ is given by the \textit{coadjoint representation} $\coad$, which is obtained from the dual vector space $H^*$ by considering the left $H$-action
\[
 (x \cdot \phi)(y) := (\coad_x(\phi))(y) = \phi(S(x_{(1)})yx_{(2)})
\]
for all $x,y \in H$ and $\phi \in H^*$, compare\footnote{Lyubashenko actually computes the coend of
a different (but isomorphic) functor,
 while we use the one of \cite[Lem.~3]{K96}.} with \cite[Thm.~3.3.1]{L94}. For every $V \in \calC$, the structure morphism $i_V : V^* \otimes V \to \coad$ is given by
\[
 i_V(\phi \otimes v)(x) := \phi(x \cdot v)
\]
for all $x \in H$, $\phi \in V^*$, and $v \in V$. Similarly, the end $\eend = \int_{V \in \calC} V \otimes V^*$ is given by the \textit{adjoint representation} $\ad$, which is obtained from the vector space $H$ by considering the left $H$-action
\[
 x \cdot y := \ad_x(y) = x_{(1)}yS(x_{(2)})
\]
for all $x,y \in H$, see \cite[Prop.~7.2]{FGR17}. For every $V \in \calC$, the structure morphism $j_V : \ad \to V \otimes V^*$ is given by
\[
 j_V(x) := \sum_{i=1}^n (x \cdot v_i) \otimes \phi^i
\]
for every $x \in H$, where $\{ v_i \in V \mid 1 \leq i \leq n \}$ and $\{ \phi^i \in V^* \mid 1 \leq i \leq n \}$ are dual bases.

The \textit{M-matrix} $M \in H \otimes H$, also known as the \textit{monodromy matrix}, is defined by
\[
 M := R_{21}R_{12},
\]
where $R_{12} = R' \otimes R''$ and $R_{21} = R'' \otimes R'$. The M-matrix $M = M' \otimes M''$ determines a \textit{Drinfeld map} $D : H^* \to H$ defined by
\[
 D(\phi) := \phi(M') M''
\]
for every $\phi \in H^*$, as first considered in \cite[Prop.~3.3]{D90}. We say $H$ is \textit{factorizable} if the Drinfeld map is an isomorphism. Remark that the Drinfeld map is an intertwiner $D \colon \coad \to \ad$ of $H$-modules, and that the categorical Drinfeld map defined in Section~\ref{sec:def_coend} coincides precisely with it, see~\cite[Rem.~7.7]{FGR17}. This means that $\calC=\mods{H}$ is modular if and only if $H$ is factorizable.

The coend $\coend$ admits the structure of a braided Hopf algebra in $\calC$, which we denote $\overline{H}$ \cite{M93}. This is given by
\begin{align*}
 \overline{\mu}(\phi \otimes \psi)(x) &= \phi(x_{(2)} R') \psi(S(R''_{(1)}) x_{(1)} R''_{(2)}), &
 \overline{\eta}(1)(x) &= \varepsilon(x), \\
 \overline{\Delta}(\phi)(x \otimes y) &= \phi(yx) &
 \overline{\epsilon}(\phi) &= \phi(1), \\
 \overline{S}(\phi)(x) &= \phi(S(R')S(x)S(u)^{-1}R'') &
 &
\end{align*}
for all $\phi,\psi \in \overline{H}$ and $x, y \in H$. The \textit{Hopf pairing} $\omega$ of \cite[Thm.~3.7]{L95} is given by
\[
 \omega(\phi \otimes \psi) := \phi(M') \psi(S(M''))
\]
for all $\phi,\psi \in H^*$.

A \textit{two-sided integral} $\intL : \Bbbk \to \coend = \coad$
is determined by a standard right integral of $H$, which means a linear form $\lambda : H \to \Bbbk$ satisfying
\[
 \lambda(x_{(1)}) x_{(2)} = \lambda(x) 1
\]
for every $x \in H$ (see \cite[Def.~10.1.2]{R12}), by setting
\[
 \intL(1) := \lambda.
\] 
Similarly, a \textit{two-sided cointegral} $\cointL : \coend \to \Bbbk$ is determined by a standard two-sided cointegral of $H$, which means an element $\lambda^\co \in H$ satisfying
\[
 x \lambda^\co = \epsilon(x) \lambda^\co = \lambda^\co x
\]
for every $x \in H$ (see \cite[Def.~10.1.1]{R12}), by setting
\[
 \cointL(\varphi) := \varphi(\lambda^\co)
\]
for every $\varphi \in H^*$.

The choice of a right integral $\lambda$ of $H$ (which is unique up to a scalar)  gives rise to a modified trace $\rmt$ on $\Proj(\calC)$ thanks to \cite[Thm.~1]{BBG18}, which is uniquely determined by
\[
 \rmt_H(f) = \lambda(g f(1))
\]
for every $f \in \End_H(H)$, where $H$ denotes the regular representation, 
and
where $g$ denotes the pivotal element defined in~\eqref{eq:pivot}.

For all $g,m \in \N$ and every $\underline{V} = (V_1,\ldots,V_m) \in \calC^{\times m}$, the algebraic pairing $\langle \_,\_ \rangle_{g,\underline{V}} : \calX'_{g,\underline{V}} \times \calX_{g,\underline{V}} \to \Bbbk$ of Section~\ref{sss:skein-alg-summary} is given by
\begin{align*}
 \langle x',x \rangle_{g,\underline{V}} 
 &= \langle x' \circ \left( (\id_{V_1 \otimes \ldots \otimes V_m} \otimes (D^{-1})^{\otimes g} \right) \circ x \rangle_{\epsilon_{\one}}
\end{align*}
for all $x \in \Hom_H(P_{\one},V_1 \otimes \ldots \otimes V_m \otimes \ad^{\otimes g})$ and $x' \in \Hom_H(V_1 \otimes \ldots \otimes V_m \otimes \coad^{\otimes g},\one)$, where for every $f \in \Hom_H(P_{\one},\one)$ the scalar $\langle f \rangle_{\epsilon_{\one}} \in \Bbbk$ is defined by $f = \langle f \rangle_{\epsilon_{\one}} \epsilon_{\one}$ for the projective cover morphism $\epsilon_{\one} : P_{\one} \to \one$. Note that \cite[Lem.~2.10]{DGGPR19} allows us to express $D^{-1} : \ad \to \coad$ as
\[
 D^{-1}(x) = \zeta^{-1}~\lambda\bigl(S^{-1}(x)S(R')S^2(M'')S(u)^{-1}R''\bigr) \cdot  \lambda(M' \_)
\]
for every $x \in \ad$, where $\zeta = \lambda(v)\lambda(v^{-1})$. Indeed, this follows from the remark that $\tilde{\ell} : \coad \to \ad^*$ is simply given by $(S^{-1})^*$.

\subsubsection{Equivalence with renormalized Hennings}

For every object $W \in \calC = \mods{H}$, let us consider the functor 
\[
 (\_^* \otimes \_)^{\times n} \otimes W : (\calC^\op \times \calC)^{\times n} \to \calC
\]
sending every object $(U_1,V_1,\ldots,U_n,V_n) \in (\calC^\op \times \calC)^{\times n}$ to the object 
\[
 U_1^* \otimes V_1 \otimes \ldots \otimes U_n^* \otimes V_n \otimes W \in \calC.
\]
If $Z \in \calC$ is another object, which can also be interpreted as a constant functor $Z : (\calC^\op \times \calC)^{\times n} \to \calC$, let us consider an $n$-dinatural transformation
\[
 \alpha : (\_^* \otimes \_)^{\times n} \otimes W \din Z,
\]
which means $\alpha$ determines a dinatural transformation
\[
 \alpha^\sigma : \left( (\_^* \otimes \_)^{\times n} \otimes W \right) \circ \sigma \din Z \circ \sigma
\]
for the permutation functor 
\[
 \sigma : (\calC^{\times n})^\op \times \calC^{\times n} \to (\calC^\op \times \calC)^{\times n}
\]
sending every object $((U_1,\ldots,U_n),(V_1,\ldots,V_n)) \in (\calC^{\times n})^\op \times \calC^{\times n}$ to the object $((U_1,V_1),\ldots,(U_n,V_n)) \in (\calC^\op \times \calC)^{\times n}$. By universality of the coend $\coend$, the transformation $\alpha$ produces a unique morphism (compare with~\cite[Sec.~3.1 \& Eq.~(26)]{DGGPR19})
\begin{equation}\label{eq:f-alpha}
 f_\calC(\alpha) : \coad^{\otimes n} \otimes W \to Z,
\end{equation}
satisfying
\[
 f_\calC(\alpha) \circ (i_{V_1} \otimes \ldots \otimes i_{V_n} \otimes \id_W) = \alpha_{(V_1,\ldots,V_n)}
\]
for all $V_1, \ldots, V_n \in \calC$. In the next lemma, we compute the unique morphism $f_\calC(\alpha)$ in the case $\calC = \mods{H}$.

\begin{lemma}\label{L:comparison_H-L}
 For $\calC = \mods{H}$, the morphism $f_\calC(\alpha)$ of~\eqref{eq:f-alpha}
 is given by
 \[
 f_\calC(\alpha) =
  \alpha_{(H,\ldots,H)} \circ \left( (\id_{H^*} \otimes \eta)^{\otimes n} \otimes \id_W \right),
 \]
 where $H$ denotes the regular representation and $\eta : \one \to H$ denotes the unit of $H$.
\end{lemma}

\begin{proof}
 For every 
 \[
  \left( \bigotimes_{k=1}^n \varphi_k \right) \otimes w \in \coad^{\otimes n} \otimes W
 \]
 and every $x \in H$ we have
 \begin{align*}
  &x \cdot \alpha_{(H,\ldots,H)} \left( \left( \bigotimes_{k=1}^n (\varphi_k \otimes 1) \right) \otimes w \right) \\*
  &\hspace*{\parindent} = \alpha_{(H,\ldots,H)} \left( \left( \bigotimes_{k=1}^n (\varphi_k(S(x_{(2k-1)}) \_) \otimes x_{(2k)}) \right) \otimes (x_{(2n+1)} \cdot w) \right) \\*
  &\hspace*{\parindent} = \alpha_{(H,\ldots,H)} \left( \left( \bigotimes_{k=1}^n (\varphi_k(S(x_{(2k-1)}) \_ x_{(2k)}) \otimes 1) \right) \otimes (x_{(2n+1)} \cdot w) \right) \\*
  &\hspace*{\parindent} = \alpha_{(H,\ldots,H)} \left( \left( \bigotimes_{k=1}^n (\coad_{x_{(k)}}(\varphi_k) \otimes 1) \right) \otimes (x_{(n+1)} \cdot w) \right),
 \end{align*}
 where the first equality follows from the fact that $\alpha_{(H,\ldots,H)} : (H^* \otimes H)^{\otimes n} \otimes W \to Z$ is an intertwiner, and the second one from the $n$-dinaturality of $\alpha$. This proves that
 \[
  \alpha_{(H,\ldots,H)} \circ \left( (\id_{H^*} \otimes \eta)^{\otimes n} \otimes \id_W \right)
 \]
 defines an intertwiner $f_\calC(\alpha) : \coad^{\otimes n} \otimes W \to Z$. Now the claim is checked by showing that 
 \[
  \alpha_{(H,\ldots,H)} \circ \left( (\id_{H^*} \otimes \eta)^{\otimes n} \otimes \id_W \right) \circ (i_{V_1} \otimes \ldots \otimes i_{V_n} \otimes \id_W) = \alpha_{(V_1,\ldots,V_n)}
 \]
 for all $V_1, \ldots, V_n \in \calC$. For every 
 \[
  \left( \bigotimes_{k=1}^n (\varphi_k \otimes v_k) \right) \otimes w \in \left( \bigotimes_{k=1}^n (V_k^* \otimes V_k) \right) \otimes W
 \]
 the $n$-dinaturality of $\alpha$ implies
 \begin{align*}
  \alpha_{(V_1,\ldots,V_n)} \left( \left( \bigotimes_{k=1}^n (\varphi_k \otimes v_k) \right) \otimes w \right) 
  &= \alpha_{(H,\ldots,H)} \left( \left( \bigotimes_{k=1}^n (\varphi_k(\_ \cdot v_k) \otimes 1) \right) \otimes w \right) \\*
  &= \alpha_{(H,\ldots,H)} \left( \left( \bigotimes_{k=1}^n \left( i_{V_k}(\varphi_k \otimes v_k) \otimes 1 \right) \right) \otimes w \right).
 \end{align*}
\end{proof}

Thanks to Lemma~\ref{L:comparison_H-L}, when $\calC = \mods{H}$ we can make the algorithm for the computation of the Lyubashenko-Reshetikhin-Turaev functor of Section~\ref{sss:LRT-functor} more explicit. Indeed, in order to evaluate $F_\intL(T)$ for a bichrome graph $T$ with $n$ red components, we can first consider an $n$-bottom graph presentation $\tilde{T}$ of $T$. Next, we can label all red components by the regular representation $H$ and forget the difference between red and blue, thus obtaining a ribbon graph $\tilde{T}_{(H,\ldots,H)}$. At this point, we can simply pre-compose $F_\calC(\tilde{T}_{(H,\ldots,H)})$ with\footnote{In general, when $T$ has incoming boundary vertices, we should actually pre-compose with the tensor product between $(\lambda \otimes 1)^{\otimes n}$ and the identity of the corresponding labels.} the $n$-fold tensor power of $\lambda \otimes 1 \in H^* \otimes H$.
This demonstrates that the renormalized Lyubashenko invariant of equation~\eqref{eq:L'-via-F'} agrees with the renormalized Hennings invariant introduced in~\cite{DGP17}. Moreover, in Appendix~\ref{A:equivalence} we prove that the TQFT construction of~\cite{DGGPR19} reviewed here recovers the one of \cite{DGP17} in the setting of Hopf algebras.

\section{Mapping class group representations from TQFTs}\label{S:MCG}

In this section, we 
    analyze
the definition of the projective mapping class group representations associated with the TQFTs $\rmV_\calC$ and $\rmV'_\calC$, and we realize the corresponding actions of Dehn twists by curve operators. 
For a choice of generators we compute the action on the algebraic state spaces $\rmX'_{g,\underline{V}}$.

\subsection{Mapping class group of a decorated surface}\label{SS:MCG}

Let us start by introducing some terminology and some notation. A \textit{decorated surface} is a pair $(\Sigma,P)$, where $\Sigma$ is a connected closed surface, and where $P \subset \Sigma$ is a blue set, that is, a finite set of oriented framed (blue) points labeled by objects of $\calC$.
If $f : \Sigma \to \Sigma'$ is a diffeomorphism of surfaces, and if $P \subset \Sigma$ is a blue set, then we denote with $f(P) \subset \Sigma'$ the blue set whose labels and framings are obtained from those of $P$ by pushforward. Then, if $(\Sigma,P)$ is a decorated surface, we consider the set of diffeomorphisms of $\Sigma$ preserving $P$,
\[
 \Diff(\Sigma,P) := \{ f \in \Diff(\Sigma) \mid f(P) = P \},
\]
which is a group with respect to composition. We denote with $\Diff_0(\Sigma,P)$ the subgroup whose elements $f$ are isotopic to $\id_\Sigma$ within $\Diff(\Sigma,P)$, and we define the \textit{mapping class group of $(\Sigma,P)$} to be the quotient group
\[
 \MCG(\Sigma,P) := \Diff(\Sigma,P) / \Diff_0(\Sigma,P).
\]
Remark that $\MCG(\Sigma,P)$ fits into the generalized Birman exact sequence
\[
 1 \to \rmB(\Sigma,P) \to \MCG(\Sigma,P) \to \MCG(\Sigma) \to 1,
\]
where $\rmB(\Sigma,P)$ denotes the colored framed braid group, which itself fits into the exact sequence
\[
 1 \to \PB(\Sigma,P) \to \rmB(\Sigma,P) \to \frakS(P) \to 1,
\]
compare with \cite[Thm.~9.1]{FM12}.
Here, $\frakS(P)$ is the colored symmetric group of all permutations of $P$ which preserve colors (i.e. the labels of the blue set), and $\PB(\Sigma,P)$ is the pure framed braid group. It is useful to fix names for a few elements of the group $\rmB(\Sigma,P)$. For all integers $1 \leq i<j \leq m$ satisfying $V_i = V_j$ we set
\begin{equation}\label{E:x_ij}
 x_{i,j} := \pic{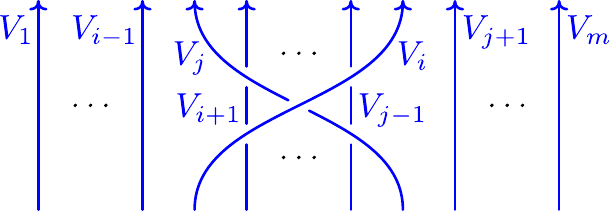} \in \rmB(\Sigma,P). 
\end{equation}
For all integers $1 \leq i<j \leq m$ we define
\begin{equation}\label{E:w_ij}
 w_{i,j} := \pic{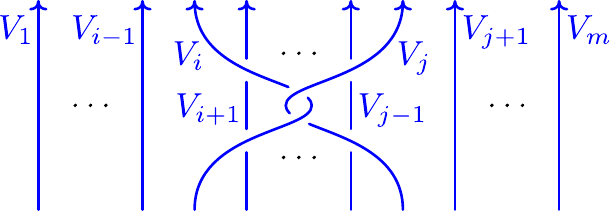} \in \PB(\Sigma,P),
\end{equation}
and finally for every integer $1 \leq i \leq m$ we set
\begin{equation}\label{E:v_i}
 v_i := \pic{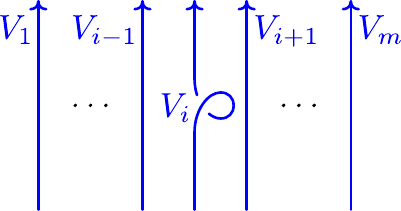} \in \PB(\Sigma,P).
\end{equation}
The Reshetikhin-Turaev functor $F_\calC$ assigns to these elements endomorphisms 
\[
 F_\calC(x_{i,j}), F_\calC(w_{i,j}), F_\calC(v_i) \in \End_\calC(V_1 \otimes \ldots \otimes V_m).
\]
For example,
\begin{align*}
F_\calC(w_{1,2}) &= (c_{V_2,V_1} \circ c_{V_1,V_2}) \otimes \id_{V_3 \otimes \cdots \otimes V_m} ,
\\
 F_\calC(v_1) &= \theta_{V_1} \otimes \id_{V_2 \otimes \cdots \otimes V_m} .
\end{align*}

\subsection{Projective representations from mapping cylinders}\label{sec:mapping_cyl}

If $\bbSigma = (\Sigma,P,\lambda)$ is a connected object of $\adCob_\calC$, then the \textit{mapping cylinder} of an element $f \in \Diff(\Sigma,P)$ is defined as the morphism $\bbSigma \times \bbI_f : \bbSigma \to \bbSigma$ of $\adCob_\calC$ given by
\[
 (\Sigma \times I_f,P \times I,0),
\]
where the cobordism $\Sigma \times I_f$ has incoming and outgoing boundary identifications
\[
 (f,0) : \Sigma \to \Sigma \times \{ 0 \}, \qquad
 (\id_\Sigma,1) : \Sigma \to \Sigma \times \{ 1 \},
\]
respectively. It is easy to check that 
\begin{equation}\label{eq:map-cyl-compos}
 \left( \bbSigma \times \bbI_{f'} \right) \circ \left( \bbSigma \times \bbI_f \right) \sim \bbSigma \times \bbI_{f' \circ f}
\end{equation}
as morphisms of $\adCob_\calC$, where $\sim$ denotes the equivalence relation determined by forgetting about signature defects. Let us stress that in general the composition~\eqref{eq:compose-Cob} will produce a non-trivial signature defect, because a mapping cylinder, as well as any other morphism of $\adCob_\calC$, is not required to preserve the Lagrangian subspaces of its source and target under push-forward and pull-back.

If $[f] = [f']$ in $\MCG(\Sigma,P)$, then any isotopy $h : \Sigma \times I \to \Sigma$ between $f' \circ f^{-1}$ and $\id_\Sigma$ determines an isomorphism between the cobordisms $\Sigma \times I_f$ and $\Sigma \times I_{f'}$. 
This means we obtain a map
\begin{align*}
 \repT : \MCG(\Sigma,P) & \to \GL_\Bbbk(\rmV_\calC(\bbSigma)) \\*
 {}[f] & \mapsto \rmV_\calC(\bbSigma \times \bbI_f). 
\end{align*}
Since \eqref{eq:map-cyl-compos} only holds up to signature defects, as a consequence of the composition rule \eqref{eq:compose-Cob}, for a pair of elements $[f],[f'] \in \MCG(\Sigma,P)$ the linear map $\repT[f' \circ f]$ may differ from $\repT[f'] \circ \repT[f]$ by a power of $\delta$, as defined in \eqref{eq:delta-def}. Therefore, if $\delta\neq 1$, the map $\repT$ is \textit{not} a group homomorphism, but instead only defines a projective representation. We denote the resulting group homomorphism by
\[
        \prepT
        :=
        \Big[ 
    \MCG(\Sigma,P) \xrightarrow{\repT} \GL_\Bbbk(\rmV_\calC(\bbSigma))
    \twoheadrightarrow \PGL_\Bbbk(\rmV_\calC(\bbSigma))
        \Big].
\]
However, one can obtain a linear representation of a central extension of the mapping class group, see Appendix~\ref{A:extended_MCG_reps}.

For $\rmV'_\calC$ we get a corresponding map
\begin{align}
 \repT' : \MCG(\Sigma,P) & \to \GL_\Bbbk(\rmV'_\calC(\bbSigma)) \label{eq:sigma'-def} \\*
 {}[f] & \mapsto \rmV'_\calC(\bbSigma \times \bbI_f)^{-1} \nonumber
\end{align}
and group homomorphism
\begin{equation}\label{eq:barsigma'-def}
 \prepT' :=
        \Big[ 
    \MCG(\Sigma,P) \xrightarrow{\repT'} \GL_\Bbbk(\rmV'_\calC(\bbSigma))
    \twoheadrightarrow \PGL_\Bbbk(\rmV'_\calC(\bbSigma))
        \Big].
\end{equation}
The inverse in \eqref{eq:sigma'-def} is needed for $\prepT'$ to be a group homomorphism, as the source category for $\rmV'_\calC$ is $(\adCob_\calC)^{\op}$.

Because of the duality \eqref{eq:V-V'-dual} between $\rmV_\calC$ and $\rmV'_\calC$ with respect to the pairing $\langle \_,\_ \rangle_{\bbSigma}$,
for every $[f] \in \MCG(\Sigma,P)$ we have
\begin{equation}\label{eq:sig-sig'-dual}
 \repT'[f] = (\repT[f]^t)^{-1}.
\end{equation}

\subsection{Dehn twists and curve operators}\label{SS:Dehn_twist_curve_op}

If $\bbSigma = (\Sigma,P,\lambda)$ is a connected object of $\adCob_\calC$, and $\gamma \subset \Sigma$ is a
 simple closed curve, then let us denote with $\tau_\gamma \in \Diff(\Sigma,P)$ the \text{Dehn twist} along $\gamma$. This is the self-diffeomorphism of $\Sigma$ restricting to the identity outside of a tubular neighborhood of $\gamma$, and given by
\begin{align*}
 S^1 \times I & \to S^1 \times I \\*
 ((\cos(\theta),\sin(\theta)),t) & \mapsto ((\cos(\theta + 2 \pi t),\sin(\theta + 2 \pi t)),t)
\end{align*}
in local coordinates around $\gamma$, up to appropriately smoothing at transitions points. There are, up to isotopy, two orientation-preserving ways to embed the above neighborhood into $\Sigma$, related by $(\vartheta,t) \mapsto (-\vartheta,-t)$. However, both of these describe the same diffeomorphism. This is the reason why it is not necessary to orient $\gamma$. Instead, we could choose to replace $2\pi t$ with $-2 \pi t$, which determines a left-handed Dehn twist, as opposed to the right-handed one defined above, see \cite[Sec.~3.1.1]{FM12}. The two possible Dehn twists thus obtained are inverse with each other.

Let $\gamma_\pm$ denote the red knot $\gamma \times \left\{ \frac 12 \right\}$ inside $\Sigma \times I$ with framing $\pm 1$ relative to the surface $\Sigma \times \left\{ \frac 12 \right\}$. The \textit{curve cylinder associated with $\gamma_\pm$} is the morphism $\bbSigma \times \bbI_{\gamma_\pm} : \bbSigma \to \bbSigma$ of $\adCob_\calC$ given by
\[
\bbSigma \times \bbI_{\gamma_\pm} :=
 \left( \Sigma \times I,(P \times I) \cup \gamma_\pm,0 \right).
\]
It is a classical remark, see for instance the proof of \cite[Thm.~2]{L62}, that there exists an isomorphism of cobordisms
\begin{equation}\label{E:mapping_cyl_eq_surgereg_cyl}
    \Sigma \times I_{\tau_\gamma^{\pm 1}} \cong (\Sigma \times I)(\gamma_\mp),
\end{equation}
where the cobordism $(\Sigma \times I)(\gamma_\mp)$ is obtained by performing 2-surgery on $\Sigma \times I$ 
along $\gamma_\mp$, with both incoming and outgoing boundary identifications induced by $\id_\Sigma$. Combining this with \cite[Prop.~4.10]{DGGPR19} we immediately get
\begin{equation}\label{E:coeff_between_mapping_and_curve_cylinder}
 \rmV_\calC(\bbSigma \times \bbI_{\tau_\gamma^{\pm 1}}) = \rmV_\calC((\Sigma \times I)(\gamma_\mp),P \times I,0) = \calD^{-1} \delta^{m_\mp} 
 \rmV_\calC(\bbSigma \times \bbI_{\gamma_\mp}).
\end{equation}
for some $m_\mp \in \Z$. Remark that this exponent is determined by the composition rule \eqref{eq:compose-Cob}, which intervenes in \cite[Prop.~4.10]{DGGPR19}.

\begin{remark}\label{R:framing_anomaly}
 If the homology class of $\gamma$ belongs to the Lagrangian $\lambda$, then $m_\mp = \pm 1$, and thus, thanks to \eqref{eq:delta-def}, we get
 \[
  \rmV_\calC(\bbSigma \times \bbI_{\tau_\gamma^{\pm 1}}) = \Delta_\mp^{-1} \rmV_\calC(\bbSigma \times \bbI_{\gamma_\mp}).
 \]
 For a proof of this, see Appendix \ref{A:Maslov}.
\end{remark}

Let us also give a slightly different argument relating $\rmV_\calC(\bbSigma \times \bbI_{\tau_\gamma^{\pm 1}})$ to $\rmV_\calC(\bbSigma \times \bbI_{\gamma_\mp})$ (and ditto for $\rmV'_\calC$) which uses algebraic properties of the TQFT rather than the topology of surgered manifolds. Denote the image of $\rmV_\calC(\bbSigma \times \bbI_{\gamma_\pm}) \in \GL_\Bbbk(\rmV_\calC(\bbSigma))$ in the projectivization by $[\rmV_\calC(\bbSigma \times \bbI_{\gamma_\pm})] \in \PGL_\Bbbk(\rmV_\calC(\bbSigma))$. 
Then:

\begin{lemma}\label{L:mapping_and_curve_cylinder}
 If $(\Sigma,P)$ is a decorated surface and $\gamma \subset \Sigma$ is a simple closed curve, then $\prepT(\tau_\gamma^{\pm 1}) = [\rmV_\calC(\bbSigma \times \bbI_{\gamma_\mp})]$ and $\prepT'(\tau_\gamma^{\pm 1}) = [\rmV'_\calC(\bbSigma \times \bbI_{\gamma_\pm})]$.
\end{lemma}

\begin{proof}
    We will only treat $\prepT(\tau_\gamma)$ and $\prepT'(\tau_\gamma)$ explicitly, since the corresponding statement for $\tau_\gamma^{-1}$ can be seen analogously.

As recalled in Section~\ref{sss:skein-alg-summary}, if $M$ is a fixed connected cobordism from $\varnothing$ to $\Sigma$, then every vector of $\rmV_\calC(\bbSigma)$ can be represented by a linear combination of bichrome graphs $T$ inside $M$ from $\varnothing$ to $P$. In choosing $M$, let us make sure $\gamma$ bounds a disc in it.
Then the Dehn twist $\tau_\gamma$ extends to a self-diffeomorphism $\tilde{\tau}_\gamma$ of the 3-manifold $M$, which defines an isomorphism between the cobordism $M$ and the cobordism $\tilde{\tau}_\gamma(M)$ whose underlying 3-manifold is still $M$, but whose outgoing boundary identification is $\tilde{\tau}_\gamma \circ f_M$, where $f_M$ is the outgoing boundary identification of $M$. This means
\[
 (M,T,0) = (\tilde{\tau}_\gamma(M),\tilde{\tau}_\gamma(T),0)
\]
as morphisms of $\adCob_\calC$. But now, identifying the outgoing boundary of $\tilde{\tau}_\gamma(M)$ to the incoming boundary of $\Sigma \times I_{\tau_\gamma}$ using the diffeomorphism 
\[
 \partial M \xrightarrow{(\tilde{\tau}_\gamma \circ f_M)^{-1}} \Sigma \xrightarrow{(\tau_\gamma,0)} \Sigma \times \{ 0 \}
\]
is the same as using
\[
 \partial M \xrightarrow{f_M^{-1}} \Sigma \xrightarrow{(\id_\Sigma,0)} \Sigma \times \{ 0 \}.
\]
This means
\[
 \tilde{\tau}_\gamma(M) \cup_\Sigma (\Sigma \times I_{\tau_\gamma}) = M \cup_\Sigma (\Sigma \times I)
\]
as cobordisms. However, the signature defect determined by the composition rule \eqref{eq:compose-Cob} for the morphism 
\[
 (\bbSigma \times \bbI_{\tau_\gamma}) \circ (\tilde{\tau}_\gamma(M),\tilde{\tau}_\gamma(T),0)
\]
might be non-trivial in general. This is why we only have
 \[
  (\bbSigma \times \bbI_{\tau_\gamma}) \circ (M,T,0) = (\bbSigma \times \bbI_{\tau_\gamma}) \circ (\tilde{\tau}_\gamma(M),\tilde{\tau}_\gamma(T),0) \sim (M,\tilde{\tau}_\gamma(T),0),
 \]
 where $\sim$ was introduced in \eqref{eq:map-cyl-compos}. But now
 \begin{align*}
  \repT(\tau_\gamma)([M,T,0]) &= [(\bbSigma \times \bbI_{\tau_\gamma}) \circ (M,T,0)] \\*
  &\propto [M,\tilde{\tau}_\gamma(T),0]
  \overset{\raisebox{2.5pt}{$(\ast)$}} \propto [M,T \cup \gamma_-,0]
  = \rmV_\calC(\bbSigma \times \bbI_{\gamma_-})([M,T,0])
 \end{align*}
as elements of $\rmV_\calC(\bbSigma)$. Step $(\ast)$ is obtained by first adding an unknotted red component of framing $-1$, at the cost of multiplying by $\Delta_-^{-1}$, and then applying the slide property of \cite[Prop.~3.7]{DGGPR19}, see Figure~\ref{F:Dehn_twist}. 

 \begin{figure}[t]
     \includegraphics{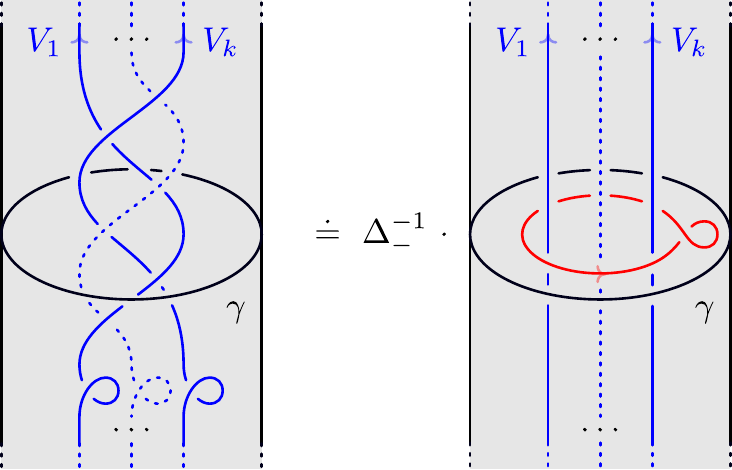}
     \caption{Skein equivalence inside a neighborhood of the disc bounded by $\gamma$ in $M$, which contains portions of edges of $T$ labeled by some objects $V_1,\ldots,V_k \in \calC$ (red edges could also appear). 
     }
     \label{F:Dehn_twist}
 \end{figure}
 
This shows the statement for $\prepT(\tau_\gamma)$, while the claim for $\prepT'(\tau_\gamma)$ follows from the duality relations \eqref{eq:V-V'-dual} and \eqref{eq:sig-sig'-dual}.
\end{proof}

\subsection{Action of generators on algebraic state spaces}\label{SS:algebraic_translation}

In this section we describe one of the main outcomes of this paper, namely the projective action of a set of generators of the mapping class group of our chosen standard surface $(\Sigma_g,P_{\underline{V}})$ from \eqref{eq:standard-surf_gV}
on the algebraic state space $\rmX'_{g,\underline{V}}$ as in \eqref{eq:alg-statespace}.

We use the isomorphisms $\Phi'$ from \eqref{eq:Phi'-def} to transport $\repT'$ as given in \eqref{eq:sigma'-def} from $\rmV'_\calC(\bbSigma_{g,\underline{V}})$ to $\rmX'_{g,\underline{V}}$. That is, we define the map 
\begin{equation}\label{eq:rho'-def}
 \repCpr : \MCG(\Sigma_g,P_{\underline{V}}) \to \GL_\Bbbk(\rmX'_{g,\underline{V}})
\end{equation}
by requiring the following diagram to commute for all $[f] \in \MCG(\Sigma,P)$:
\begin{center}
 \begin{tikzpicture}
  \node (P1) at (150:1.5) {$\rmX'_{g,\underline{V}}$};
  \node (P2) at (30:1.5) {$\rmX'_{g,\underline{V}}$};
  \node (P3) at (-150:1.5) {$\rmV'_\calC(\bbSigma_{g,\underline{V}})$};
  \node (P4) at (-30:1.5) {$\rmV'_\calC(\bbSigma_{g,\underline{V}})$};
  \draw
  (P1) edge[->] node[left] {\scriptsize $\Phi'$} (P3)
  (P1) edge[->] node[above] {\scriptsize $\repCpr(f)$} (P2)
  (P2) edge[->] node[right] {\scriptsize $\Phi'$} (P4)
  (P3) edge[->] node[below] {\scriptsize $\repT'(f)$} (P4);
 \end{tikzpicture}
\end{center}
As for $\repT'$, in general only the projectivization
\begin{equation}\label{eq:prho'-def}
 \prepCpr : \MCG(\Sigma_g,P_{\underline{V}}) \to \PGL_\Bbbk(\rmX'_{g,\underline{V}})
\end{equation}
of  $\repCpr$ is a group homomorphism.

Note that while the underlying vector space $\rmX'_{g,\underline{V}}=\calC(V_1 \otimes \ldots \otimes V_m \otimes \coend^{\otimes g},\one)$ and the group homomorphism $\prepCpr$ only depend on $\calC$, the original map $\repCpr$ also depends on the choice of the Lagrangian subspace of the standard object $\bbSigma_{g,\underline{V}}$.

We now turn to the explicit computation of $\repCpr$ for a set of generators of the mapping class group. We start with a few preliminary definitions which closely follow \cite{L94}.
The \textit{monodromy} $\Omega : \coend \otimes \coend \to \coend \otimes \coend$ is the unique morphism of $\calC$ satisfying, for all $X,Y \in \calC$, the identity
\begin{equation}\label{E:monodromy}
 \Omega \circ (i_X \otimes i_Y) = (i_X \otimes i_Y) \circ \left( \id_{X^*} \otimes (c_{Y^*,X} \circ c_{X,Y^*}) \otimes \id_Y \right).
\end{equation}
Similarly, if $X$ is an object of $\calC$, the \textit{left partial monodromy} $\Omega_{\rmL,X} : X \otimes \coend \to X \otimes \coend$ is uniquely determined, for every $Y \in \calC$, by
\begin{equation}\label{E:left_partial_monodromy}
 \Omega_{\rmL,X} \circ (\id_X \otimes i_Y) = (\id_X \otimes i_Y) \circ \left( (c_{Y^*,X} \circ c_{X,Y^*}) \otimes \id_Y \right),
\end{equation}
and if $Y$ is an object of $\calC$, the \textit{right partial monodromy} $\Omega_{\rmR,Y} : \coend \otimes Y \to \coend \otimes Y$ is uniquely determined, for every $X \in \calC$, by
\begin{equation}\label{E:right_partial_monodromy}
 \Omega_{\rmR,Y} \circ (i_X \otimes \id_Y) = (i_X \otimes \id_Y) \circ \left( \id_{X^*} \otimes (c_{Y,X} \circ c_{X,Y}) \right).
\end{equation}
Next, the \textit{S-} and \textit{T-transformation} $\calS, \calT : \coend \to \coend$ are uniquely specified, for every $X \in \calC$, by
\begin{align}
\label{E:S}
 \calS \circ i_X &= (\counitL \otimes \id_\coend) \circ \Omega \circ (i_X \otimes \intL),
 \\
\label{E:T}
 \calT \circ i_X &= i_X \circ (\id_{X^*} \otimes \theta_X).
\end{align}
A graphical presentation of the above equalities as \textit{skein equivalences} is given in Figure~\ref{fig:bichrome-preliminary-defs} using $n$-bottom graphs (possibly up to pre-composing with braidings to reorder red vertices to the left), which evaluate to dinatural transformations as outlined in Section~\ref{sss:LRT-functor} (or to morphisms in $\calC$ if $n=0$). The symbol $\doteq$ means that both sides evaluate to the same dinatural transformation (respectively morphism). For example, both sides of the skein equivalence for $\Omega_{\rmL,X}$ evaluate to the same family of morphisms $\{ \eta_Y : X \otimes Y^* \otimes Y \to X \otimes \coend \}_{Y \in \calC}$ which is dinatural in $Y$. This definition of skein equivalence is slightly more general than the one used in \cite[Sec.~4.2]{DGGPR19}, where only blue boundary vertices were allowed, and no $n$-bottom graphs with $n>0$. Note that, in order to arrive at the bichrome presentation of $\calS$, one needs to use that the counit $\counitL$ and the integral $\intL$ of $\coend$ can be expressed as a red cap and cup respectively.

\begin{figure}[b]
    \centering
 \includegraphics{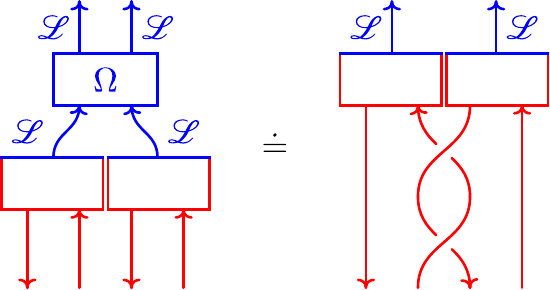} \\[1em]
 \includegraphics{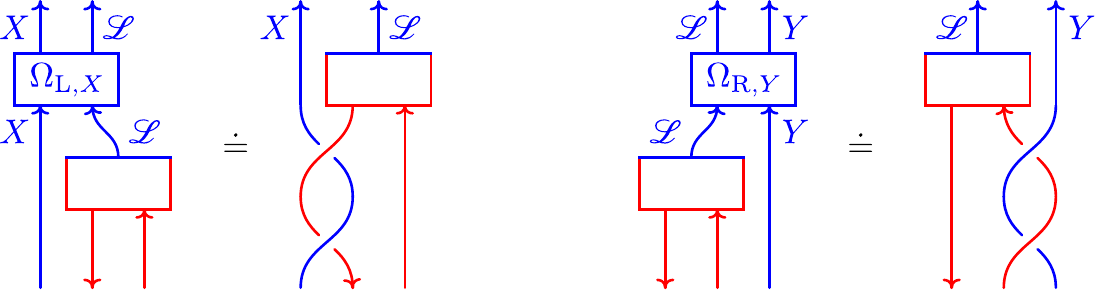}\\[1em]
 \includegraphics{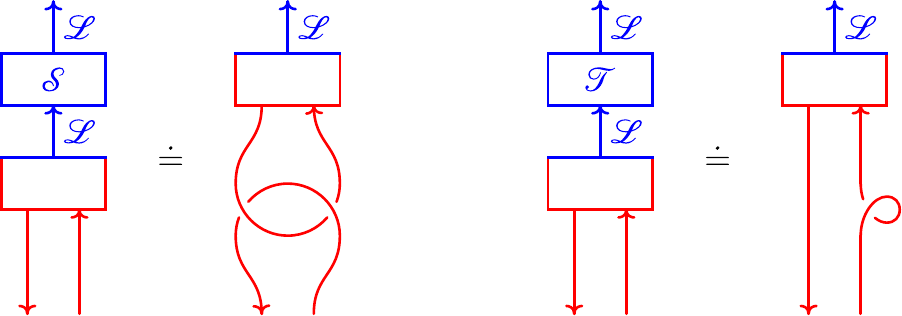}
\caption{Graphical representation of the definitions in \eqref{E:monodromy}--\eqref{E:T}.}
    \label{fig:bichrome-preliminary-defs}
\end{figure}

It will also be useful to fix names for $\calH : \coend \otimes \coend \to \coend \otimes \coend$ defined by
\begin{equation}\label{E:H}
 \calH := \Omega \circ (\calT \otimes \calT),
\end{equation}
as well as for its left and right partial versions $\calH_{\rmL,X} : X \otimes \coend \to X \otimes \coend$ and $\calH_{\rmR,Y} : \coend \otimes Y \to \coend \otimes Y$ defined, for all $X,Y \in \calC$, by
\begin{align}
 \calH_{\rmL,X} &:= \Omega_{\rmL,X} \circ (\vartheta_X \otimes \calT), \label{E:H_L} \\*
 \calH_{\rmR,Y} &:= \Omega_{\rmR,Y} \circ (\calT \otimes \vartheta_Y). \label{E:H_R}
\end{align}

\begin{figure}[t]
\[
 \includegraphics{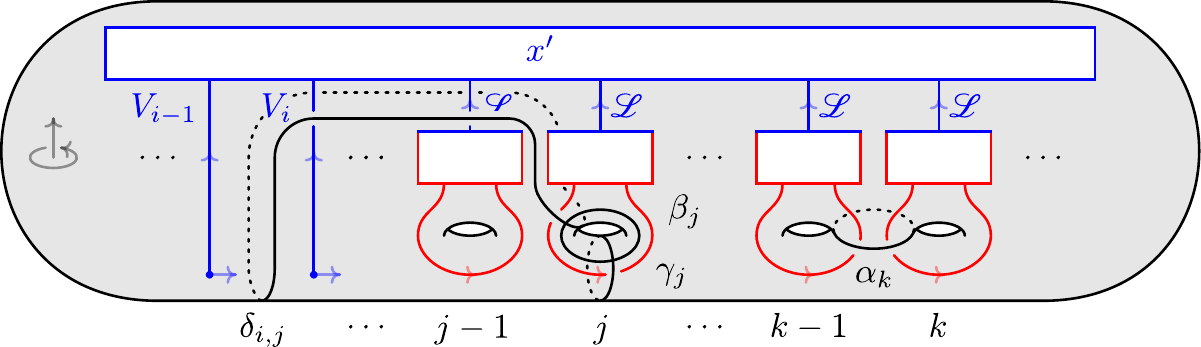}
\]
\caption{Simple closed curves used in the description of generators of the mapping class group: $\alpha_k, \beta_j, \gamma_j, \delta_{i,j} \subset \Sigma_g$ for all integers $1 \leq i \leq m$, $1 \leq j \leq g$, and $2 \leq k \leq g$.}
\label{fig:curves-for-generators}
\end{figure}

As explained in \cite[Sec.~4.5]{L94}, and using the notation in Section~\ref{SS:MCG} and Figure~\ref{fig:curves-for-generators}, the mapping class group $\MCG(\Sigma_g,P_{\underline{V}})$ is generated by:
\begin{itemize}
 \item $x_{i,j}$ for all integers $1 \leq i < j \leq m$ satisfying $V_i = V_j$;
 \item $w_{i,j}$ for all integers $1 \leq i < j \leq m$;
 \item $v_i$ for every integer $1 \leq i \leq m$;
 \item $H_k := \tau_{\alpha_k}^{-1}$ for every integer $2 \leq k \leq g$;
 \item $S_j := \tau_{\gamma_j} \circ \tau_{\beta_j} \circ \tau_{\gamma_j}$ for every integer $1 \leq j \leq g$;
 \item $T_j := \tau_{\gamma_j}^{-1}$ for every integer $1 \leq j \leq g$;
 \item $H_{i,j} := \tau_{\delta_{i,j}}^{-1}$ for all integers $1 \leq i \leq m$, $1 \leq j \leq g$.
\end{itemize}
Lyubashenko considers some additional generators corresponding to inverse Dehn twists around simple closed curves which are homologous to $\gamma_j$ for all $1 \leq j \leq g$. However, these extra generators are redundant, as follows from \cite[Lem.~5]{L64}. 

\begin{proposition}\label{P:equivalence}
The map
 $\repCpr : \MCG(\Sigma_g,P_{\underline{V}}) \to \GL_\Bbbk(\rmX'_{g,\underline{V}})$ satisfies
 \begin{align*}
  \repCpr(x_{i,j})(x') &= x' \circ \left( F_\calC(x_{i,j})^{-1} \otimes \id_{\coend^{\otimes g}} \right) \\*
  \repCpr(w_{i,j})(x') &= x' \circ \left( F_\calC(w_{i,j})^{-1} \otimes \id_{\coend^{\otimes g}} \right) \\*
  \repCpr(v_i)(x') &= x' \circ \left( F_\calC(v_i)^{-1} \otimes \id_{\coend^{\otimes g}} \right) \\*
  \repCpr(H_k)(x') &= \delta^{a_k} \cdot x' \circ (\id_{V_1 \otimes \ldots \otimes V_m \otimes \coend^{\otimes k-2}} \otimes \calH \otimes \id_{\coend^{\otimes g-k}}) \\*
  \repCpr(S_j)(x') &= \calD^{-1} \delta^{b_j} \cdot x' \circ (\id_{V_1 \otimes \ldots \otimes V_m \otimes \coend^{\otimes j-1}} \otimes \calS \otimes \id_{\coend^{\otimes g-j}}) \\*
  \repCpr(T_j)(x') &= \delta^{c_j} \cdot x' \circ (\id_{V_1 \otimes \ldots \otimes V_m \otimes \coend^{\otimes j-1}} \otimes \calT \otimes \id_{\coend^{\otimes g-j}}) \\*
  \repCpr(H_{i,j})(x') &= \delta^{d_{i,j}} \cdot x' \circ \left( \id_{V_1 \otimes \ldots \otimes V_{i-1}} \otimes \calH_{\rmL,V_i \otimes \ldots \otimes V_m \otimes \coend^{\otimes j-1}} \otimes \id_{\coend^{\otimes g-j}} \right)
 \end{align*}
 for every $x' \in \rmX'_{g,\underline{V}} = \calC(V_1 \otimes \ldots \otimes V_m \otimes \coend^{\otimes g},\one)$, and for some $a_k,b_j,c_j,d_{i,j} \in \mathbb{Z}$.
\end{proposition}

\begin{proof}
 The action of generators $x_{i,j}$, $w_{i,j}$, and $v_i$ follows directly from the definitions. Remark that no coefficient appears here, because all these mapping classes restrict to the identity outside of a disc containing $P_{\underline{V}}$, and therefore they all induce the identity endomorphism of $H_1(\Sigma_g;\R)$. 
 
 \begin{figure}[t]
  \includegraphics{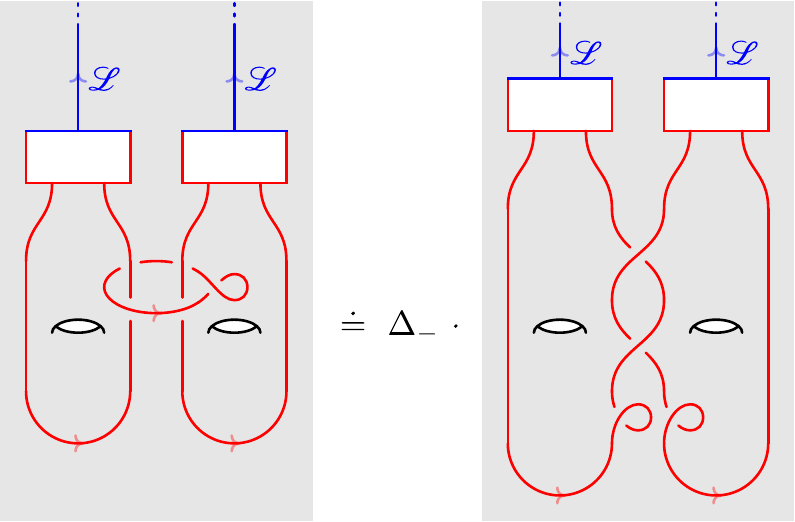}
  \caption{Skein equivalence for the action of $H_k$.}
  \label{F:equivalence_H_k}
 \end{figure}
 
 For what concerns $H_k$, we have the skein equivalence represented in Figure~\ref{F:equivalence_H_k}.
 Here, the left hand side represents the vector $\bbSigma \times \bbI_{(\alpha_k)_-} \circ \Phi'(x')$, which is proportional to $\Phi'(\repCpr(H_k)(x'))$ thanks to Lemma \ref{L:mapping_and_curve_cylinder}. Remark that the inverse in $H_k = \tau_{\alpha_k}^{-1}$ cancels with the one in the definition \eqref{eq:sigma'-def} of $\repT'$. The form of the proportionality coefficient follows from equation \eqref{E:coeff_between_mapping_and_curve_cylinder}, and from the fact that $\Delta_\pm = \calD \delta^{\pm 1}$, see~\eqref{eq:delta-def}. 

 Similarly, to compute the action of $S_j$, we recall its definition as a composition of three Dehn twists, and we use the skein equivalences of Figure~\ref{F:equivalence_S_j}.
 We point out that the second ``$\dot =$'' there
is not actually a skein equivalence as defined in \cite[Sec.~4.2]{DGGPR19}, but rather a non-local version of the notion that takes place inside a solid torus, and that uses the edge slide property \cite[Prop.~3.7, Fig.~26]{DGGPR19}. Remark also that because of the absence of inverses in the definition of $S_j$ the red knots on the left hand side are $+1$-framed. The residual factor of $\calD^{-1}$ in the proportionality coefficient follows from the fact that the number of Dehn twists required by $S_j$ is three, while the number of red knots removed by the first skein equivalence is only two.

 \begin{figure}[t]
  \includegraphics{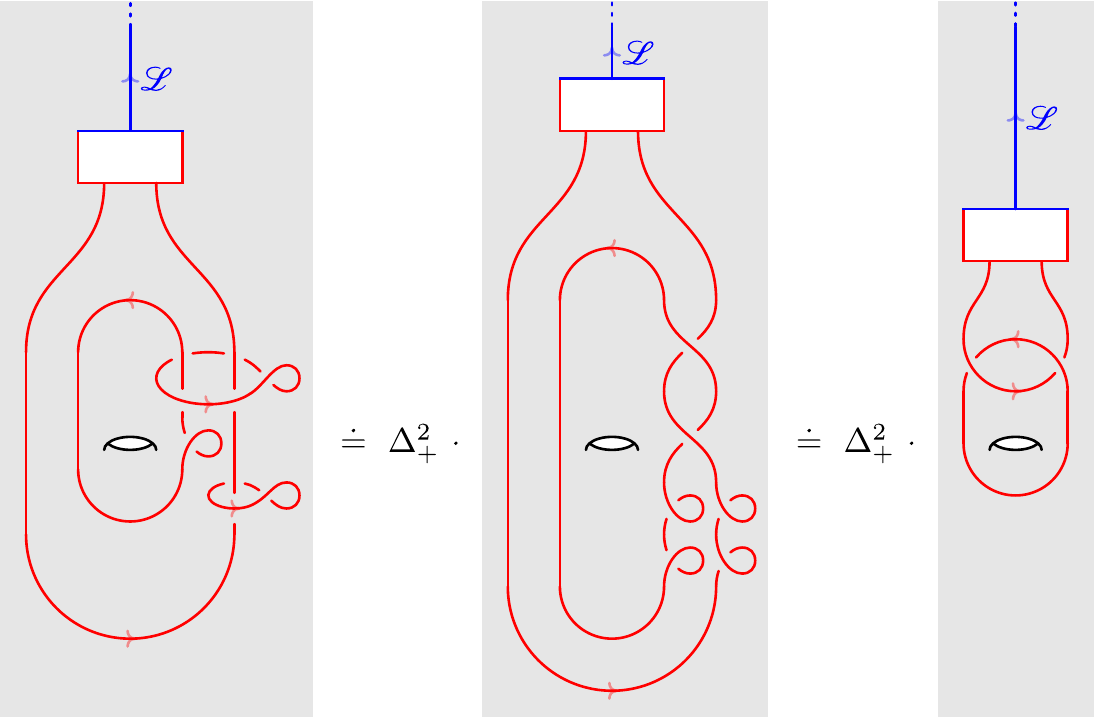}
  \caption{Skein equivalence for the action of $S_j$.}
  \label{F:equivalence_S_j}
 \end{figure}
 
For $T_j$, we consider the skein equivalence of Figure~\ref{F:equivalence_T_j}, and for $H_{i,j}$ the one of Figure~\ref{F:equivalence_H_ij}. \qedhere

 \begin{figure}[t]
  \includegraphics{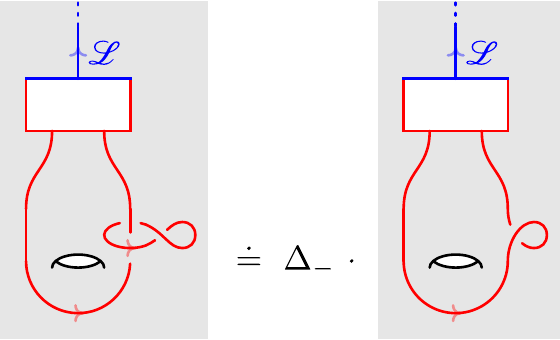}
  \caption{Skein equivalence for the action of $T_j$.}
  \label{F:equivalence_T_j}
 \end{figure}
 
 \begin{figure}[t]
  \includegraphics{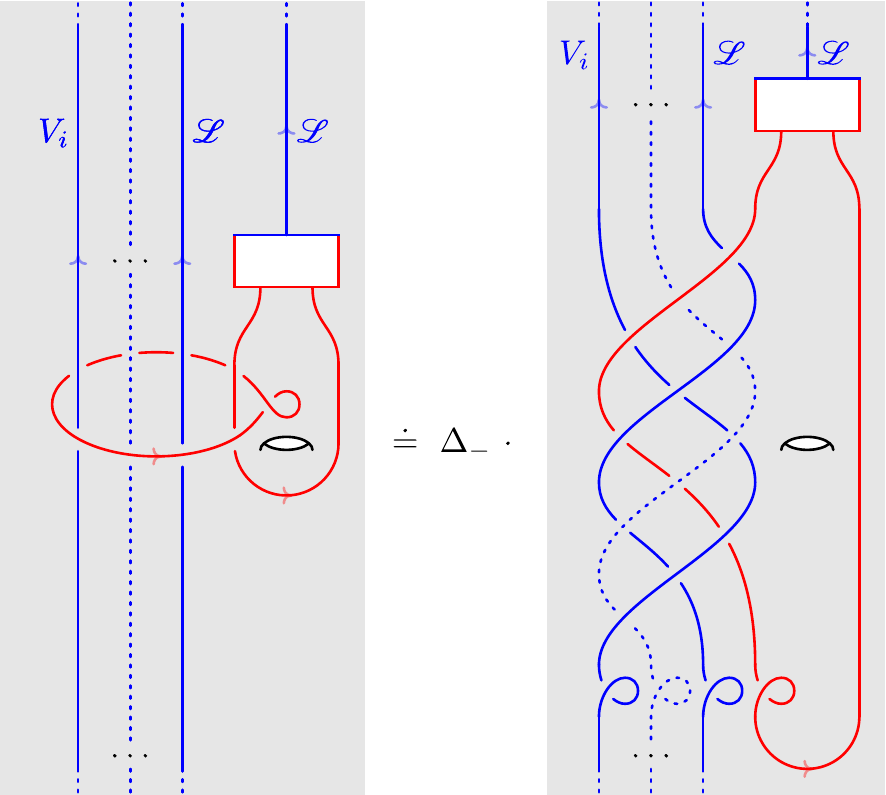}
  \caption{Skein equivalence for the action of $H_{i,j}$.}
  \label{F:equivalence_H_ij}
 \end{figure}
 
\end{proof}

\FloatBarrier

\begin{remark}
 The integers $a_k,b_j,c_j,d_{i,j}$ in the statement of Proposition \ref{P:equivalence} arise from the composition rule \eqref{eq:compose-Cob} in $\adCob_\calC$, and thus depend on the choice of the Lagrangian subspace $\lambda_g \subset H_1(\Sigma_g;\R)$ of $\bbSigma_{g,\underline{V}}$. They do not affect the projectivization $\prepCpr$, so we will not compute them here. 
 Remark however that if $M'_g$ denotes the cobordism from $\Sigma_g$ to $\varnothing$ corresponding to the top handlebody in Figure \ref{F:algebraic_model}, and if $\lambda_g = \ker (i_g)_*$ for the embedding $i_g : \Sigma_g \hookrightarrow M'_g$ induced by the structure map of $M'_g$, then $a_k = c_j = d_{i,j} = 0$. This follows immediately from Remark \ref{R:framing_anomaly}, and from the proof of Proposition~\ref{P:equivalence}.
\end{remark}

\section{Equivalence with Lyubashenko's projective representations}\label{S:equivalence}

In this section we show that the projective representation $\prepCpr$ defined via TQFT in \eqref{eq:barsigma'-def} is equivalent to the representation given in Lyubashenko's original work \cite{L94}. 
Lyubashenko considers the vector spaces
\[
 \rmL_{g,\underline{V}} := \calC(V_1 \otimes \ldots \otimes V_m,\coend^{\otimes g})
\]
and defines a group homomorphism
\[
 \prepL : \MCG(\Sigma_g,P_{\underline{V}}) \to \PGL(\rmL_{g,\underline{V}}).
\]
After briefly recalling the construction of $\prepL$, we will give an explicit isomorphism $\rmX'_{g,\underline{V}} \to \rmL_{g,\underline{V}}$ and show that it intertwines the two projective actions. This is our second main result.

\subsection{Lyubashenko's representations}\label{SS:Lyubashenko}

For all objects $V,X,Y \in \calC$ we consider the isomorphisms
\begin{align*}
 \eviso{V} : \calC(X,Y \otimes V^*) & \to \calC(X \otimes V,Y) \\*
 f & \mapsto (\id_Y \otimes \lev_V) \circ (f \otimes \id_V) \\*
 \coeviso{V} : \calC(X \otimes V,Y) & \to \calC(X,Y \otimes V^*) \\*
 f & \mapsto (f \otimes \id_{V^*}) \circ (\id_X \otimes \lcoev_V)
\end{align*}
induced by the pivotal structure of $\calC$. Lyubashenko's representation is defined in terms of generators, and we start by recalling from \cite[Sec.~4]{L94} the linear map $\repL(f)$ assigned to each generator $f$ of $\MCG(\Sigma_g,P_{\underline{V}})$ as given in Section \ref{SS:algebraic_translation}. For every $\ell \in \rmL_{g,\underline{V}}$,
\begin{align}
  \repL(x_{i,j})(\ell) &:= \ell \circ F_\calC(x_{i,j})^{-1} \nonumber \\*
  \repL(w_{i,j})(\ell) &:= \ell \circ F_\calC(w_{i,j})^{-1} \nonumber \\*
  \repL(v_i)(\ell) &:= \ell \circ F_\calC(v_i)^{-1} \nonumber \\*
  \repL(H_k)(\ell) &:= (\id_{\coend^{\otimes g-i}} \otimes \calH \otimes \id_{\coend^{\otimes k-2}}) \circ \ell \label{eq:Lyu-generators} \\*
  \repL(S_j)(\ell) &:= (\id_{\coend^{\otimes g-j}} \otimes \calS \otimes \id_{\coend^{\otimes j-1}}) \circ \ell \nonumber \\*
  \repL(T_j)(\ell) &:= (\id_{\coend^{\otimes g-j}} \otimes \calT \otimes \id_{\coend^{\otimes j-1}}) \circ \ell \nonumber \\*
  \repL(H_{i,j})(\ell) &:= \eviso{V_i \otimes \ldots \otimes V_m} \left( (\id_{\coend^{\otimes g-j}} \otimes \calH_{\rmR,\coend^{\otimes j-1} \otimes V_m^* \otimes \ldots \otimes V_i^*}) \circ \coeviso{V_i \otimes \ldots \otimes V_m} (\ell) \right) \nonumber
\end{align}
Remark that Lyubashenko actually considers inverse braiding and inverse twist morphisms, which is why our formulas for the first three kinds of generators are inverse with respect to those in \cite[Sec.~4]{L94}. We also point out that Lyubashenko works with a bigger group, since he allows mapping classes which only preserve $P_{\underline{V}}$ as a set, but not as a $\calC$-colored one. Consequently, he considers more general braiding morphisms, which allow him to list just two kinds of framed braid generators. Since we only consider mapping classes which preserve $\calC$-colorings, we have restricted Lyubashenko's representation accordingly. In terms of these linear maps, it is shown in \cite{L94,L96} by verifying the relevant relations that:

\begin{theorem}
There is a unique homomorphism $\prepL : \MCG(\Sigma_g,P_{\underline{V}}) \to \PGL(\rmL_{g,\underline{V}})$ which satisfies $\prepL(f) = [\repL(f)]$, where $f$ runs over the set of generators in \eqref{eq:Lyu-generators}.
\end{theorem}

\begin{remark}
 Since Lyubashenko describes the projective representation $\prepL$ in terms of generators $f$ of $\MCG(\Sigma_g,P_{\underline{V}})$, he only defines the corresponding linear endomorphisms $\repL(f) \in \GL(\rmL_{g,\underline{V}})$, and one does not obtain a specific choice for an extension to a complete lift $\repL : \MCG(\Sigma_g,P_{\underline{V}}) \to \GL(\rmL_{g,\underline{V}})$ of $\prepL$. Indeed, different realizations of an arbitrary element of $\MCG(\Sigma_g,P_{\underline{V}})$ as a combination of generators would give linear maps that differ by scalars. However, in Section~\ref{SS:algebraic_translation} we saw that the TQFT approach naturally provides (and even starts from) such lifts, and that different choices of Lagrangian subspaces lead to different lifts.
\end{remark}

\subsection{Radford copairing}\label{SS:copairing}

In order to give the isomorphism between $\rmX'_{g,\underline{V}}$ and $\rmL_{g,\underline{V}}$ we will use the \textit{Radford copairing} $\Rad : \one \to \coend \otimes \coend$, which is defined as
\begin{equation}\label{E:Radford_copairing}
 \Rad := \copL \circ \intL.
\end{equation}
As we have seen in Section~\ref{sss:LRT-functor}, in terms of bichrome graphs the definition of the Radford copairing amounts to the skein equivalence
\begin{equation}\label{eq:R-bichrome}
 \pic{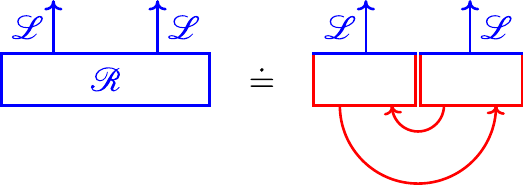}
\end{equation}
We use the shorthand notation $\Rad^{(n)} : \one \to \coend^{\otimes n} \otimes \coend^{\otimes n}$ for the morphism of $\calC$ defined inductively by
\[
 \Rad^{(0)} := \id_{\one}, \qquad \Rad^{(n)} := (\id_\coend \otimes \Rad^{(n-1)} \otimes \id_\coend) \circ \Rad.
\]
Similarly, for all $X,Y \in \calC$ and every $f \in \calC(\one,X \otimes Y)$, we use the abbreviation $\Rad^{(n)}_f : \one \to X \otimes \coend^{\otimes n} \otimes \coend^{\otimes n} \otimes Y$ for the morphism 
\[
 \Rad^{(n)}_f := (\id_X \otimes \Rad^{(n)} \otimes \id_Y) \circ f.
\]
The following properties of the Radford copairing will be needed to prove the equivalence of projective representations.

\begin{lemma}\label{L:Radford_prop}
The Radford copairing is non-degenerate and satisfies
 \begin{align}
  (\Omega \otimes \id_{\coend^{\otimes 2}}) \circ \Rad^{(2)} &= 
  (\id_{\coend^{\otimes 2}} \otimes \Omega) \circ \Rad^{(2)} \label{E:Radford_prop_monodromy} \\*
  (\Omega_{\rmL,X} \otimes \id_{\coend \otimes Y}) \circ \Rad_f &= 
  (\id_{X \otimes \coend} \otimes \Omega_{\rmR,Y}) \circ \Rad_f \label{E:Radford_prop_partial_monodromy} \\*
  (\calS \otimes \id_\coend) \circ \Rad &= 
  (\id_\coend \otimes \calS) \circ \Rad \label{E:Radford_prop_S} \\*
  (\calT \otimes \id_\coend) \circ \Rad &= 
  (\id_\coend \otimes \calT) \circ \Rad \label{E:Radford_prop_T} 
 \end{align}
 for all $X,Y \in \calC$ and every $f \in \calC(\one,X \otimes Y)$.
\end{lemma}

\begin{proof}
Non-degeneracy follows from \cite[Cor.~4.2.13]{KL01}. 

Next consider Equation~\eqref{E:Radford_prop_monodromy}.
Using the skein equivalence for $\Omega$ of Figure~\ref{fig:bichrome-preliminary-defs} and the one for $\Rad$ of \eqref{eq:R-bichrome}, the two sides of the identity in \eqref{E:Radford_prop_monodromy} are skein equivalent to the two sides of
 \[
  \pic{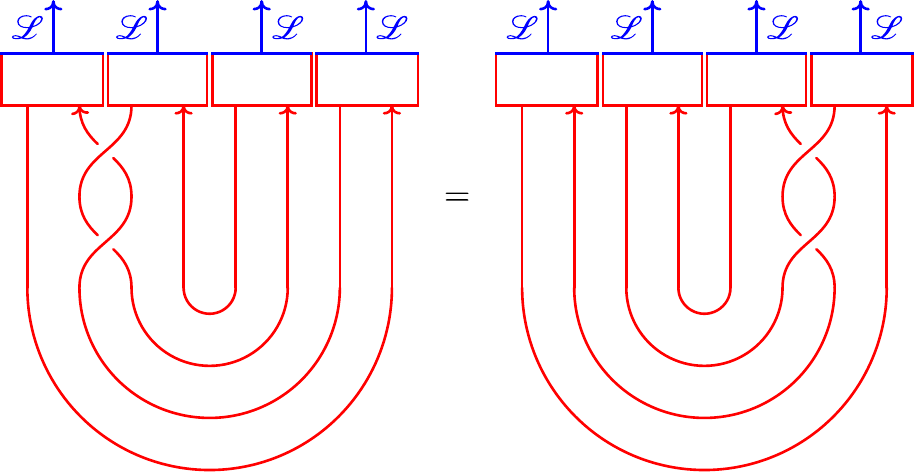}
 \]
which are equal to each other, as they are related by an isotopy. All other identities are shown in the same way. We only state the relevant skein equivalences and isotopies.
\begin{itemize}
    \item 
  Equation \eqref{E:Radford_prop_partial_monodromy}:
\[
  \pic{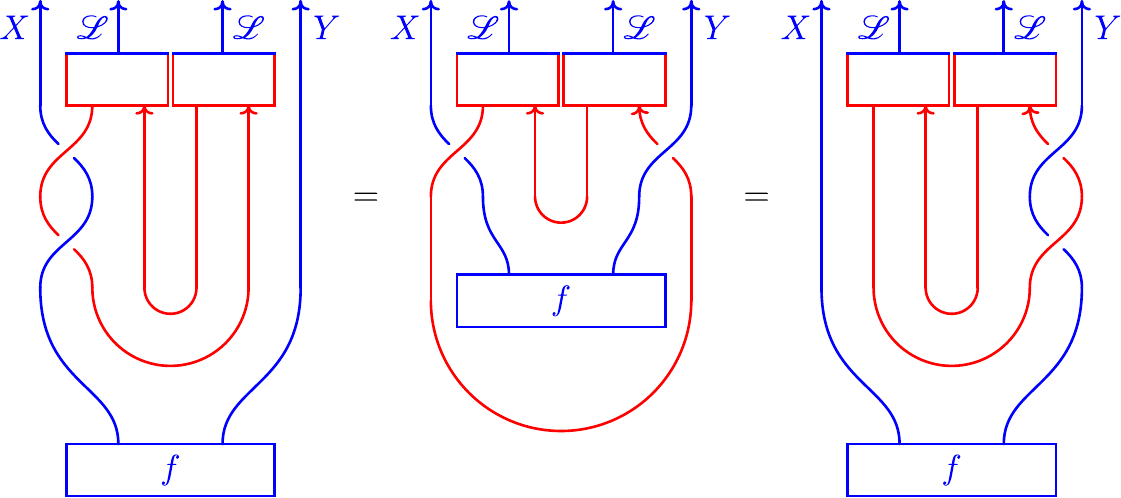}
 \]
\item 
 Equation \eqref{E:Radford_prop_S}:
 \[
  \pic{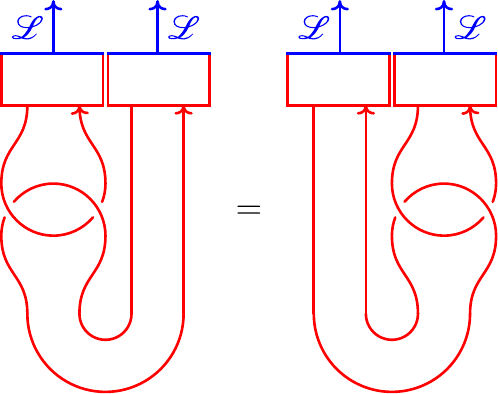}
 \]
\item
 Equation \eqref{E:Radford_prop_T}:
 \[
  \pic{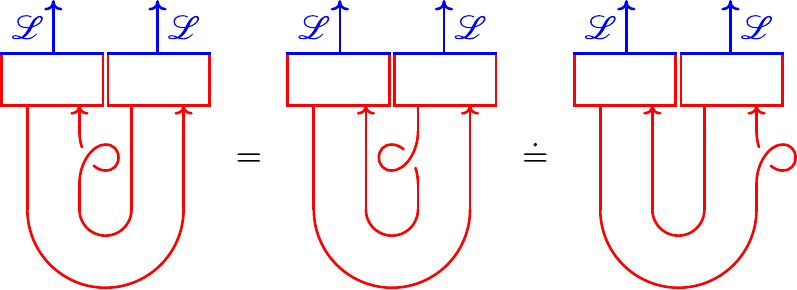} 
 \]
In the last step we used dinaturality of the bichrome coupon. \qedhere
\end{itemize}
\end{proof}

\subsection{Equivalence of representations}\label{SS:main_result}

Consider the linear map
\begin{align*}
 \phi : \rmX'_{g,\underline{V}} & \to \rmL_{g,\underline{V}} \\*
 x' & \mapsto (x' \otimes \id_{\coend^{\otimes g}})
 \circ (\id_{V_1 \otimes \ldots \otimes V_m} \otimes \Rad^{(g)}).
\end{align*}
Since the copairing is non-degenerate (see Lemma~\ref{L:Radford_prop}), the map $\phi$ is invertible. Let us write $\phi^{\mathrm{ad}}$ for the induced group isomorphism
\begin{align*}
 \phi^{\mathrm{ad}} : \PGL(\rmX'_{g,\underline{V}}) & \to \PGL(\rmL_{g,\underline{V}}) \\*
 {}[f] & \mapsto [\phi \circ f \circ \phi^{-1}].
\end{align*}

\begin{theorem}\label{T:main-equiv}
The isomorphism $\phi$ intertwines the projective representations $\prepCpr$ and $\prepL$, that is, $\phi^{\mathrm{ad}}$ fits into the commutative diagram
 \begin{center}
  \begin{tikzpicture}[descr/.style={fill=white}]
  \node (P1) at (0:0) {$\MCG(\Sigma_g,P_{\underline{V}})$};
  \node (P2) at (15:3) {$\PGL(\rmX'_{g,\underline{V}})$};
  \node (P3) at (-15:3) {$\PGL(\rmL_{g,\underline{V}})$};
  \draw
  (P1) edge[->] node[above] {\scriptsize $\prepCpr$} (P2)
  (P1) edge[->] node[below] {\scriptsize $\prepL$} (P3)
  (P2) edge[->] node[right] {\scriptsize $\phi^{\mathrm{ad}}$} (P3);
 \end{tikzpicture}
\end{center}
\end{theorem}

\begin{proof}
 It is enough to show that for each generator $f$ of $\MCG(\Sigma_g,P_{\underline{V}})$ used in Proposition~\ref{P:equivalence} and \eqref{eq:Lyu-generators} we have
\[
    \phi \circ \repCpr(f) \propto \repL(f) \circ \phi.
\]
For the actions of generators $x_{i,j}$, $w_{i,j}$, $v_i$, this is immediate. For the others, we use Lemma \ref{L:Radford_prop}. Let $x' \in \rmX'_{g,\underline{V}}$ be arbitrary. Using  \eqref{E:Radford_prop_monodromy} and \eqref{E:Radford_prop_T}, we see
 \begin{align*}
  \phi \left( \repCpr(H_k)(x') \right)
  & \propto
 \phi \left( x' \circ (\id_{V_1 \otimes \ldots \otimes V_m \otimes \coend^{\otimes k-2}} \otimes \calH \otimes \id_{\coend^{\otimes g-k}}) \right) \\*
  & = (\id_{\coend^{\otimes g-i}} \otimes \calH \otimes \id_{\coend^{\otimes k-2}}) \circ \phi(x') \\*
  & = \repL(H_k) 
  \left( \phi(x') \right).
 \end{align*}
The computation for generators $S_j$ and $T_j$ is analogous:
 \begin{align*}
 \phi \left( \repCpr(S_j)(x') \right)
  & \propto
   \phi \left( x' \circ (\id_{V_1 \otimes \ldots \otimes V_m \otimes \coend^{\otimes j-1}} \otimes \calS \otimes \id_{\coend^{\otimes g-j}}) \right) \\*
  & \overset{\eqref{E:Radford_prop_S}}= (\id_{\coend^{\otimes g-j}} \otimes \calS \otimes \id_{\coend^{\otimes j-1}}) \circ \phi(x') \\*
  & \overset{\phantom{\eqref{E:Radford_prop_S}}} = \repL(S_j) \left( \phi(x') \right)
\\
\phi \left( \repCpr(T_j)(x') \right) 
  & \propto
  \phi \left( x' \circ (\id_{V_1 \otimes \ldots \otimes V_m \otimes \coend^{\otimes j-1}} \otimes \calT \otimes \id_{\coend^{\otimes g-j}}) \right) \\*
  & \overset{\eqref{E:Radford_prop_T}}= (\id_{\coend^{\otimes g-j}} \otimes \calT \otimes \id_{\coend^{\otimes j-1}}) \circ \phi(x') \\*
  &\overset{\phantom{\eqref{E:Radford_prop_T}}} = \repL(T_j) \left( \phi(x') \right)
 \end{align*}
Finally, for the generator $H_{i,j}$ we use \eqref{E:Radford_prop_T} as well as \eqref{E:Radford_prop_partial_monodromy}  with
 \[
  X = V_i \otimes \ldots \otimes V_m \otimes \coend^{\otimes j-1}, \quad
  Y = \coend^{\otimes j-1} \otimes V_m^* \otimes \ldots \otimes V_i^*, \quad
  f = \Rad^{(j-1)}_{\stackrel{\longleftarrow}{\mathrm{coev}}_{V_i \otimes \ldots \otimes V_m}}.
 \]
Then 
 \begin{align*}
  &\phi \left( \repCpr(H_{i,j})(x') \right) \\*
  &\hspace*{\parindent} \propto \phi \left( x' \circ \left( \id_{V_1 \otimes \ldots \otimes V_{i-1}} \otimes \calH_{\rmL,V_i \otimes \ldots \otimes V_m \otimes \coend^{\otimes j-1}} \otimes \id_{\coend^{\otimes g-j}} \right) \right) \\*
  &\hspace*{\parindent} = \eviso{V_i \otimes \ldots \otimes V_m} \left( (\id_{\coend^{\otimes g-j}} \otimes \calH_{\rmR,\coend^{\otimes j-1} \otimes V_m^* \otimes \ldots \otimes V_i^*}) \circ \coeviso{V_i \otimes \ldots \otimes V_m} (\phi(x')) \right) \\*
  &\hspace*{\parindent} = \repL(H_{i,j}) \left( \phi(x') \right),
 \end{align*}
where we used the naturality of the twist to deduce that
 \[
  (\theta_X \otimes \id_Y) \circ f = (\id_X \otimes \theta_Y) \circ f. \qedhere
 \]
\end{proof}

\begin{remark}
 When $\calC = \mods{H}$
 for a factorizable ribbon Hopf algebra $H$, 
 explicit formulas for Lyubashenko's projective representations can be found in several places. The case
 of genus
  $g = 1$ was first discussed in \cite[Thm.~4.4]{LM94} (using the adjoint representation instead of the coadjoint one) and in \cite[Sec.~2.5]{K94}, see also \cite[Rem.~8.3]{FGR17} for the corresponding equations in our conventions. The higher genus case $g > 1$ was studied in \cite[Sec.~4]{FSS12} (using the category $\bimods{H}$ instead of $\mods{H}$) and in \cite[Thm.\,5.12]{F18}.
 \end{remark}

\section{Dehn twists for the small quantum group}\label{S:properties}

In this section, we consider the particular modular category $\calC = \mods{\bar{U}_q \sl_2}$ coming from the representation theory of the small quantum group $\bar{U}_q \sl_2$, and we prove that the action of Dehn twists on state spaces of closed surfaces without decorations has infinite order. This is in contrast with the Reshetikhin-Turaev TQFT for the semisimple modular category obtained as a subquotient of $\calC$, as the action of all Dehn twists has finite order there. A closely related observation on the infinite order of Dehn twists was made in \cite{BCGP14} for the non-semisimple graded TQFTs obtained from the so-called \textit{unrolled} quantum group $U^H_q \sl_2$.

\subsection{Small quantum group}\label{SS:small-quantum}

Let us set $q = e^{\frac{2 \pi i}{r}}$ for some odd integer $r \geq 3$. For every natural number $k \in \N$ we introduce the notation 
\[
 \{ k \} := q^k - q^{-k},
 \quad [k] := \frac{\{ k \}}{\{ 1 \}},
 \quad [k]! := \prod_{j=1}^k [j].
\]
The small quantum group $\bar{U}_q \sl_2$, first defined in \cite{L90}, can be constructed as the $\C$-algebra with generators $\{ E,F,K \}$ and relations
\begin{gather*}
 E^r = F^r = 0, \qquad K^r = 1, \\*
 K E K^{-1} = q^2 E, \qquad K F K^{-1} = q^{-2} F, \qquad [E,F] = \frac{K - K^{-1}}{q-q^{-1}},
\end{gather*}
and with Hopf algebra structure obtained by 
setting\footnote{Lusztig considers the opposite coproduct, while we are using the one of \cite[Ex.~3.4.3]{M95}.}
\begin{align*}
 \Delta(E) &= E \otimes K + 1 \otimes E, & \varepsilon(E) &= 0, & S(E) &= -E K^{-1}, \\*
 \Delta(F) &= K^{-1} \otimes F + F \otimes 1, & \varepsilon(F) &= 0, & S(F) &= - K F, \\*
 \Delta(K) &= K \otimes K, & \varepsilon(K) &= 1, & S(K) &= K^{-1}.
\end{align*}
 A basis of $\bar{U}_q \sl_2$ is given by
\[
 \left\{ E^a F^b K^c \mid \ 0 \leq a,b,c \leq r - 1 \right\},
\]
as proved in \cite[Thm.~5.6]{L90}. Furthermore, the Hopf algebra $\bar{U}_q \sl_2$ supports a ribbon structure. Indeed, an R-matrix $R \in \bar{U}_q \sl_2 \otimes \bar{U}_q \sl_2$ and its inverse are given by
\begin{align*}
 R &= \frac{1}{r} \sum_{a,b,c=0}^{r-1} \frac{\{ 1 \}^a}{[a]!}
 q^{\frac{a(a-1)}{2} - 2bc} K^b E^a \otimes K^c F^a, \\*
 R^{-1} &= \frac{1}{r} \sum_{a,b,c=0}^{r-1} \frac{\{ -1 \}^a}{[a]!}
 q^{-\frac{a(a-1)}{2} + 2bc} E^a K^b \otimes F^a K^c,
\end{align*}
while a ribbon element $v \in \bar{U}_q \sl_2$ and its inverse are given by
\begin{align}
 v &= \frac{i^{\frac{r-1}{2}}}{\sqrt{r}} \sum_{a,b = 0}^{r - 1} \frac{\{ -1 \}^a}{[a]!} q^{-\frac{a(a-1)}{2} + \frac{(r+1)(a-b-1)^2}{2}} F^a K^b E^a, \nonumber \\
 v^{-1} &= \frac{i^{-\frac{r-1}{2}}}{\sqrt{r}} \sum_{a,b = 0}^{r - 1} \frac{\{ 1 \}^a}{[a]!} q^{\frac{a(a-1)}{2} + \frac{(r-1)(a+b-1)^2}{2}} F^a K^b E^a. \label{eq:ribbon-Uq}
\end{align}
The formulas for $v$ and $v^{-1}$ can be obtained by adapting the proof of \cite[Thm.~4.1.1]{FGST05}. A pivotal element $g \in \bar{U}_q \sl_2$ which is compatible with the ribbon structure is given by $g := K$, as explained in \cite[Prop.~XIV.6.5]{K95}. Compare with \cite[Sec.~3]{K94}, where the term \textit{balancing element} is used, and remark that Kerler considers a different antipode, which explains why he obtains the inverse of our pivotal element. A \textit{right integral} $\lambda : \bar{U}_q \sl_2 \to \C$ is given by
\[
 \lambda \left( E^a F^b K^c \right) := \frac{r^3}{\{ 1 \}^{2r-2}} \delta_{a,r-1} \delta_{b,r-1} \delta_{c,1},
\]
see \cite[Prop.~A.5.1]{L94}, and a two-sided cointegral $\lambda^\co \in \bar{U}_q \sl_2$ is given 
by\footnote{Lyubashenko uses Lusztig's coproduct, which explains why our formula defines a right integral, instead of a left one. }
\[
 \lambda^\co := \frac{\{ 1 \}^{2r-2}}{r^3} \sum_{a=0}^{r-1} E^{r-1} F^{r-1} K^a,
\]
see \cite[Prop.~A.5.2]{L94}.
It is proved in \cite[Cor.~A.3.3]{L94} that the ribbon Hopf algebra $\bar{U}_q \sl_2$ is factorizable. This can also be checked directly by looking at the M-matrix $M \in \bar{U}_q \sl_2 \otimes \bar{U}_q \sl_2$ given by
\[
 M = \frac{1}{r} \sum_{a,b,c,d=0}^{r-1} \frac{\{ 1 \}^{a+b}}{[a]![b]!} q^{\frac{a(a-1)+b(b-1)}{2} - 2 cd - (b+c)(b-d)} F^b K^c E^a E^b K^d F^a,
\]
see for instance \cite[Ex.~3.4.3]{M95}. Now it follows that the category of finite-dimensional representations $\calC = \mods{\bar{U}_q \sl_2}$ is a modular category. A direct computation gives the \textit{stabilization parameters}
\begin{align*}
 \Delta_- &= \lambda(v) = i^{\frac{r-1}{2}} r^{\frac 32} q^{\frac{r+3}{2}}, &
 \Delta_+ &= \lambda(v^{-1}) = i^{-\frac{r-1}{2}} r^{\frac 32} q^{\frac{r-3}{2}},
\end{align*}
which determine the \textit{modularity parameter}
\[
 \zeta := \Delta_- \Delta_+ = r^3.
\]
We fix the square root
\[
 \calD = r^{\frac 32}
\]
of $\zeta$, which uniquely determines the coefficient
\[
 \delta = i^{-\frac{r-1}{2}} q^{\frac{r-3}{2}},
\]
as well as the corresponding normalization of $\rmL'_\calC$. The projective cover $P_{\one}$ of the trivial $\bar{U}_q \sl_2$-module $\one = \C$ is the indecomposable projective $\bar{U}_q \sl_2$-module with basis $\{ a_0, x_k, y_k, b_0 \in P_{\one} \mid 0 \leq k \leq r-2 \}$ and left $\bar{U}_q \sl_2$-action given, for every integer $0 \leq k \leq r-2$, by
\begin{align*}
 K \cdot a_0 &= a_0, \\*
 E \cdot a_0 &= 0, \\*
 F \cdot a_0 &= 0, \\*
 K \cdot x_k &= q^{-2k-2} x_k, \\*
 E \cdot x_k &= -[k][k+1] x_{k-1}, \\*
 F \cdot x_k &= 
 \begin{cases}
  x_{k+1} & 0 \leq k < r-2, \\
  a_0 & k = r-2,
 \end{cases} \\*
 K \cdot y_k &= q^{-2k-2} y_k, \\*
 E \cdot y_k &= 
 \begin{cases}
  a_0 & k = 0, \\
  -[k][k+1] y_{k-1} & 0 < k \leq r-2,
 \end{cases} \\*
 F \cdot y_k &= y_{k+1}, \\*
 K \cdot b_0 &= b_0, \\*
 E \cdot b_0 &= x_{r-2} \\*
 F \cdot b_0 &= y_0 
\end{align*}
where $x_{-1} := y_{r-1} := 0$, compare with \cite[Sec.~C.2]{FGST05}. We denote with $h \in \End_\calC(P_{\one})$ the nilpotent endomorphism given by 
\begin{equation}\label{eq:h-def}
 h(a_0) = h(x_k) = h(y_k) = 0, \qquad h(b_0) = a_0
\end{equation}
for every integer $0 \leq k \leq r-2$, and we remark that the non-degenerate modified trace $\rmt$ determined by $\lambda$ satisfies
\[
 \rmt_{P_{\one}}(h) \neq 0.
\]
This follows from the fact that $\End_\calC(P_{\one})$ is linearly generated by $\id_{P_{\one}}$ and $h$, and that $h^2 = 0$.

\subsection{Infinite order of Dehn twist actions}\label{SS:infinite-order}

We are now ready to prove the announced property.

\begin{proposition}\label{P:infinite_order_Dehn_twists}
 If $\calC = \mods{\bar{U}_q \sl_2}$, $\bbSigma = (\Sigma,\varnothing,\lambda)$ is a connected object of $\adCob_\calC$, and $\gamma \subset \Sigma$ is an essential simple closed curve, then $\prepT(\tau_\gamma)$ has infinite order in $\PGL_\C(\rmV_\calC(\bbSigma))$.
\end{proposition}

\begin{proof}
 Let us first suppose $\gamma$ is non-separating. Then there exists a simple closed curve $\beta \subset \Sigma$ intersecting $\gamma$ exactly once. Let $M$ be a connected cobordism from $\varnothing$ to $\Sigma$ such that $\gamma$ bounds a disc in $M$, and let us consider the vector $[M,T_\beta,0] \in \rmV_\calC(\bbSigma)$, where $T_\beta$ is the blue knot obtained from $\beta$ by choosing framing tangent to $\Sigma$ and label $P_{\one}$, and by pushing the result in the interior of $M$. Also, let $\dot{T}_\beta$ denote the blue graph obtained from $T_\beta$ by adding a blue coupon labeled by the endomorphism $h$ of~$P_{\one}$ from equation~\eqref{eq:h-def}.
A computation based on the skein equivalence of Figure~\ref{F:Dehn_twist} and the action of 
 the inverse\footnote{Recall our convention on ribbon twist in Sec.~\ref{subsec:Hopf}.} ribbon element~\eqref{eq:ribbon-Uq} on $P_{\one}$ shows that
\begin{align}
   \rmV_\calC(\bbSigma \times \bbI_{\tau_\gamma})[M,T_\beta,0] &
    = \delta^m \cdot \Big( [M,T_\beta,0] + (r-1)(q-q^{-1})[M,\dot{T}_\beta,0] \Big), 
\label{eq:example-non-separating_1}
\\*
 \rmV_\calC(\bbSigma \times \bbI_{\tau_\gamma})[M,\dot{T}_\beta,0] &
    = \delta^m \cdot
[M,\dot{T}_\beta,0],
\label{eq:example-non-separating_2}
 \end{align}
 for some $m \in \mathbb{Z}$. We will now check that $[M,T_\beta,0]$ and $[M,\dot{T}_\beta,0]$ are linearly independent. This implies that $\rmV_\calC(\bbSigma \times \bbI_{\tau_\gamma})$ acts as a Jordan block of rank $2$ on the span of these two vectors, and hence has infinite order even in $\PGL_\Bbbk(\rmV_\calC(\bbSigma))$.

  \begin{figure}[t]
  \includegraphics{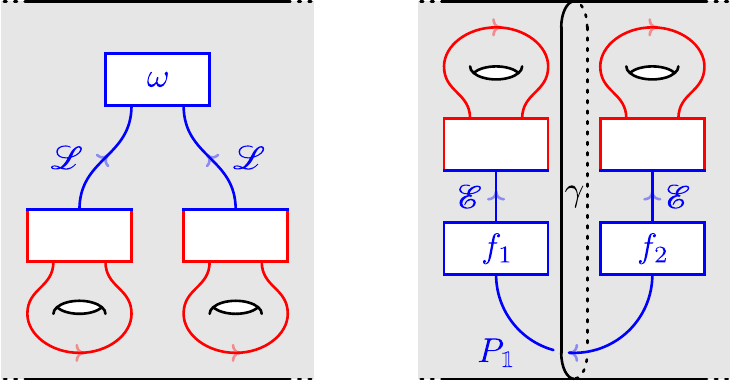}
  \caption{Bichrome graphs $T'_\omega$ and $T_{f_1,f_2}$.}
  \label{F:separating_separated}
 \end{figure}
 \begin{figure}[t]
  \includegraphics{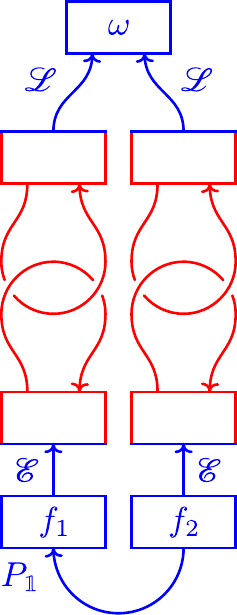}
  \caption{Bichrome graph $\dot{T}_{f_1,f_2} \cup T'_\omega$.}
  \label{F:separating_pairing}
 \end{figure}
 
To show linear independence, we first note that $[M,\dot{T}_\beta,0]$ is a non-trivial vector of $\rmV_\calC(\bbSigma)$. Indeed, if $M'$ is a cobordism from $\Sigma$ to $\varnothing$ such that $M \cup_\Sigma M' \cong S^3$, then, in the notation of equation~\eqref{eq:L'-on-cobordisms}, for some $n \in \Z$ we have
 \[
  \langle [M',\varnothing,0],[M,\dot{T}_\beta,0] \rangle_{\bbSigma} =
  \rmL'_\calC(S^3,\dot{T}_\beta,n) = \delta^n \rmt_{P_{\one}}(h) \neq 0.
 \]
This means $[M,\dot{T}_\beta,0] \neq 0$. If $[M,T_\beta,0]$ were zero, then \eqref{eq:example-non-separating_1} would imply that $[M,\dot{T}_\beta,0] = 0$, which is a contradiction. If $[M,T_\beta,0]$ were non-zero, but proportional to $[M,\dot{T}_\beta,0]$, then \eqref{eq:example-non-separating_1} and \eqref{eq:example-non-separating_2} would assign to $[M,T_\beta,0]$ two different eigenvalues. Thus $[M,T_\beta,0]$ and $[M,\dot{T}_\beta,0]$ are linearly independent.

 Next, let us suppose $\gamma$ is separating. Then there exist non-empty submanifolds $\Sigma_1, \Sigma_2 \subset \Sigma$ of strictly positive genus with disjoint interior and common boundary~$\gamma$. Let $M$ be a connected cobordism from $\varnothing$ to $\Sigma$ such that $\gamma$ bounds a disc in $M$, and let us consider the vector $[M,T_{f_1,f_2},0] \in \rmV_\calC(\bbSigma)$, where $T_{f_1,f_2}$ is the bichrome graph represented in the right-hand side of Figure \ref{F:separating_separated} for morphisms
 \[
  f_1 := \iota, \qquad 
  f_2 := \zeta^{-2} \cdot D \circ \phi_\omega \circ D^* \circ \pi^*,
 \] 
 where $\iota : P_{\one} \to \eend$ and $\pi : \eend \to P_{\one}$ are injection and projection morphisms satisfying $\pi \circ \iota = \id_{P_{\one}}$, where $D : \coend \to \eend$ is the Drinfeld map, where 
 $\phi_\omega : \coend^* \to \coend$
  is the isomorphism satisfying $\omega \circ (\id_\coend \otimes \phi_\omega) = \rev_\coend$, and where $\zeta$ is the modularity parameter. Remark that $\iota$ and $\pi$ exist because, as explained in \cite{O95}, $P_{\one}$ is a direct summand of multiplicity $\frac{r+1}{2}$ in $\eend$, which is the adjoint representation of $\bar{U}_q \sl_2$. On the other hand, the existence of $\phi_\omega$ follows from the non-degeneracy of $\omega$.  Note also that, while it is possible to simplify the expression for $f_2$, this more complicated form is convenient for the computation below.
Let $\dot{T}_{f_1,f_2}$ denote the bichrome graph obtained from $\dot{T}_{f_1,f_2}$ by adding a blue coupon with label $h$,
the endomorphism of~$P_{\one}$. Similarly to the calculation in~\eqref{eq:example-non-separating_1},
 one can check that for some $m \in \mathbb{Z}$,
 \begin{align*}
  \rmV_\calC(\bbSigma \times \bbI_{\tau_\gamma})[M,T_{f_1,f_2},0] 
  &
    = \delta^m \cdot \Big( [M,T_{f_1,f_2},0] + (r-1)(q-q^{-1})[M,\dot{T}_{f_1,f_2},0] \Big), \\*
 \rmV_\calC(\bbSigma \times \bbI_{\tau_\gamma})[M,\dot{T}_{f_1,f_2},0] &
    = \delta^m \cdot
 [M,\dot{T}_{f_1,f_2},0].
 \end{align*}
As before it is enough to check that  $[M,\dot{T}_{f_1,f_2},0] \neq 0$ in order to conclude that $\rmV_\calC(\bbSigma \times \bbI_{\tau_\gamma})$ has infinite order in $\PGL_\Bbbk(\rmV_\calC(\bbSigma))$.
 If $M'$ is a cobordism from $\Sigma$ to $\varnothing$ such that $M \cup_\Sigma  M' \cong S^3$, then for some $n \in \Z$ we have
 \begin{align*}
  \langle [M',T'_\omega,0],[M,\dot{T}_{f_1,f_2},0] \rangle_{\bbSigma} 
  &= \rmL'_\calC(S^3,\dot{T}_{f_1,f_2} \cup T'_\omega,n) 
  \\
  & \overset{\raisebox{2.5pt}{$(\ast)$}} =
  \delta^n \rmt_{P_{\one}}(h) \neq 0,
 \end{align*}
 where $T'_\omega$ is the bichrome graph represented in the left-hand side of Figure \ref{F:separating_separated}, and where $\dot{T}_{f_1,f_2} \cup T'_\omega$ is the one represented in Figure \ref{F:separating_pairing}. 
    In step $(\ast)$ we use that by \cite[Lem.~4.3, Cor.~4.6]{DGGPR19}, the inverse of the Drinfeld map $D : \coend \to \eend$ can be written as an evaluation of a bichrome graph, namely
\[
 F_\intL \left( \pic{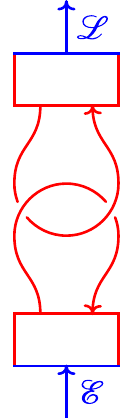} \right) = 
 \zeta D^{-1} .
\]
A cutting presentation of the admissible closed bichrome graph represented in Figure~\ref{F:separating_pairing} is then evaluated by $F_\intL$ to
\[
 \includegraphics{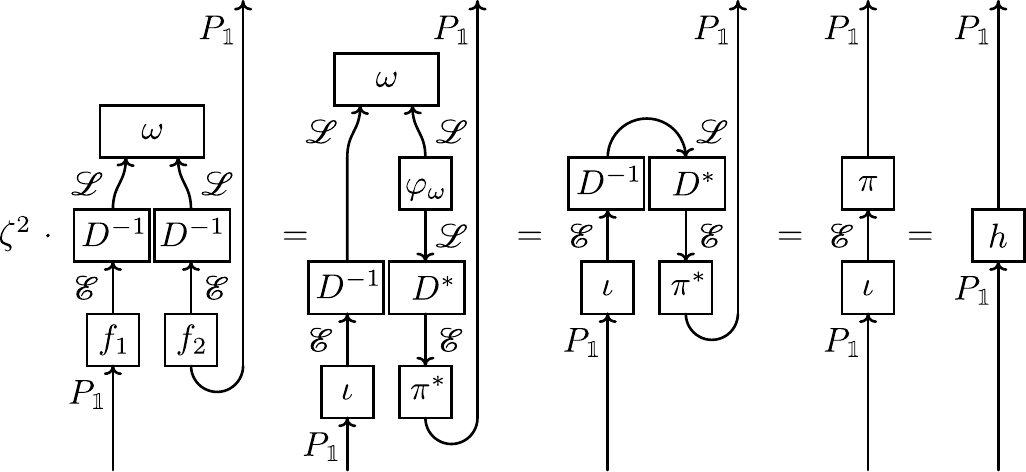} \qedhere
\]
\end{proof}

\appendix

\section{Extended mapping class group representations}\label{A:extended_MCG_reps}

In this appendix, we explain how to upgrade the projective representations
\[
 \prepT : \MCG(\bbSigma) \to \PGL_\Bbbk(\rmV_\calC(\bbSigma)), \quad
 \prepT' : \MCG(\bbSigma) \to \PGL_\Bbbk(\rmV'_\calC(\bbSigma))
\]
to linear representations of a centrally extended version of mapping class groups. This central extension is determined by the composition law for morphisms of the admissible cobordism category $\adCob_\calC$.

If $\bbSigma = (\Sigma,P,\lambda)$ is a connected object of $\adCob_\calC$, then we consider
\[
 \eDiff(\bbSigma) := \{ (f,n) \in \Diff(\Sigma) \times \Z \mid f(P) = P \},
\]
which is a group with respect to the extended composition
\[
 (f',n') \circ (f,n) := (f' \circ f,n+n'-\mu(f_*(\lambda),\lambda,{f'}^{-1}_*(\lambda))),
\]
we denote with $\eDiff_0(\bbSigma)$ the subgroup whose elements $(f,0)$ are isotopic to $(\id_\Sigma,0)$ within $\eDiff(\bbSigma)$, and we define the \textit{extended mapping class group of $\bbSigma$} to be the quotient group
\[
 \eMCG(\bbSigma) := \eDiff(\bbSigma) / \eDiff_0(\bbSigma).
\]

The \textit{mapping cylinder} of an element $(f,n) \in \eDiff(\bbSigma)$ is defined as the morphism $\bbSigma \times \bbI_{(f,n)} : \bbSigma \to \bbSigma$ of $\adCob_\calC$ given by
\[
 (\Sigma \times I_f,P \times I,n).
\]
It is easy to remark that 
\[
 \left( \bbSigma \times \bbI_{(f',n')} \right) \circ \left( \bbSigma \times \bbI_{(f,n)} \right) = \bbSigma \times \bbI_{(f',n') \circ (f,n)}
\]
as morphisms of $\adCob_\calC$. This means we have a group homomorphism 
\[
 \erepT : \eMCG(\bbSigma) \to \GL_\Bbbk(\rmV_\calC(\bbSigma))
\]
sending every extended mapping class $[f,n] \in \eMCG(\bbSigma)$ to the operator 
\[
 \rmV_\calC(\bbSigma \times \bbI_{(f,n)}) \in \GL_\Bbbk(\rmV_\calC(\bbSigma)),
\]
and similarly we have a homomorphism 
\[
 \erepT' : \eMCG(\bbSigma) \to \GL_\Bbbk(\rmV'_\calC(\bbSigma))
\]
sending every extended mapping class $[f,n] \in \eMCG(\bbSigma)$ to the operator 
\[
 \rmV'_\calC(\bbSigma \times \bbI_{(f,n)})^{-1} \in \GL_\Bbbk(\rmV'_\calC(\bbSigma)).
\]
These representations satisfy 
\[
 \erepT' \cong (\erepT^t)^{-1}
\]
with respect to the pairing $\langle \_,\_ \rangle_{\bbSigma}$.

\section{Maslov index for meridional Dehn twists}\label{A:Maslov}

In this appendix, we fix a connected object $\bbSigma = (\Sigma,P,\lambda)$ of $\adCob_\calC$, a simple closed curve $\gamma \subset \Sigma$, and we compare linear endomorphisms of the state space $\rmV_\calC(\bbSigma)$ determined by two different endomorphisms of $\bbSigma$ introduced in Section \ref{SS:Dehn_twist_curve_op}, the mapping cylinder $\bbSigma \times \bbI_{\tau_\gamma^{\pm 1}}$ and the curve cylinder $\bbSigma \times \bbI_{\gamma_\mp}$.

\begin{lemma}\label{L:Maslov}
 If the homology class of $\gamma$ belongs to the Lagrangian subspace $\lambda$ of $H_1(\Sigma;\R)$, then
 \[
  \rmV_\calC(\bbSigma \times \bbI_{\tau_\gamma^{\pm 1}}) = \Delta_\mp^{-1} \rmV_\calC(\bbSigma \times \bbI_{\gamma_\mp}).
 \]
\end{lemma}

\begin{proof}
  In \cite[Sec.~4.4]{DGGPR19}, a connected object $\bbSigma_2$ of $\adCob_\calC$, together with morphisms $\bbA_2, \bbB_2 : \varnothing \to \bbSigma_2$, is introduced. Let us recall their definition here. First, we have
  \[
   \bbSigma_2 = (S^1 \times S^1,\varnothing,\lambda_m),
  \]
  where $\lambda_m \subset H_1(S^1 \times S^1;\R)$ is the Lagrangian subspace generated by the homology class $m$ of the meridian $\{ (1,0) \} \times S^1$. Next, we have
  \[
   \bbA_2 = (S^1 \times \overline{D^2},K,0),
  \]
  where the overline stands for orientation reversal, and where $K \subset S^1 \times \overline{D^2}$ is the red knot $S^1 \times \{ (0,0) \}$ with framing determined by the longitude $S^1 \times \{ (1,0) \}$. Finally, we have
  \[
   \bbB_2 =
  (D^2 \times S^1,\varnothing,0).
  \]
  
  We will use these morphisms to relate $\bbSigma \times \bbI_{\tau_\gamma^{\pm 1}}$ and $\bbSigma \times \bbI_{\gamma_\mp}$. In order to do this, we first need to fix some notation. Let $N(\gamma_\pm)$ denote an open tubular neighborhood of the red knot $\gamma_\pm$ inside $\Sigma \times I$. Let $(\Sigma \times I) \smallsetminus N(\gamma_\pm)$ denote the corresponding cobordism from $(S^1 \times S^1) \sqcup \Sigma$ to $\Sigma$, with incoming boundary identification sending the meridian $\{ (1,0) \} \times S^1$ of $S^1 \times S^1$ to the meridian of $\partial N(\gamma_\pm)$ determined by $N(\gamma_\pm)$, and the longitude $S^1 \times \{ (1,0) \}$ of $S^1 \times S^1$ to the longitude of $\partial N(\gamma_\pm)$ determined by the framing of $\gamma_\pm$. Then, let $(\bbSigma \times \bbI) \smallsetminus \bbN(\gamma_\pm) : \bbSigma_2 \disjun \bbSigma \to \bbSigma$ denote the morphism of $\adCob_\calC$ given by
  \[
   ((\Sigma \times I) \smallsetminus N(\gamma_\pm),P \times I,0).
  \]
  
  Because of the composition rule \eqref{eq:compose-Cob}, computing $((\bbSigma \times \bbI) \smallsetminus \bbN(\gamma_\pm)) \circ \bbA_2$ and $((\bbSigma \times \bbI) \smallsetminus \bbN(\gamma_\pm)) \circ \bbB_2$ requires understanding Maslov indices of Lagrangian subspaces of $H_1(S^1 \times S^1;\R)$. Recall that, if $\omega$ is a symplectic form on a vector space $H$, and if $\lambda_1,\lambda_2,\lambda_3 \subset H$ are Lagrangian subspaces with respect to $\omega$, then $\mu(\lambda_1,\lambda_2,\lambda_3)$ is defined as the signature of the symmetric bilinear form $\langle \_,\_ \rangle$ on $(\lambda_1 + \lambda_2) \cap \lambda_3$ determined by
  \[
   \langle a,b \rangle := \omega(a_2,b)
  \]
  for all $a = a_1 + a_2, b = b_1 + b_2 \in (\lambda_1 + \lambda_2) \cap \lambda_3$, where $a_i,b_i \in \lambda_i$ for all $i \in \{ 1,2 \}$. As explained in \cite[Sec.~IV.3.5]{T94}, the Maslov index $\mu$ is completely antisymmetric in all its entries. In our case, the symplectic intersection form $\pitchfork$ on the 2-dimensional space $H_1(S^1 \times S^1;\R)$ is completely determined by $\ell \pitchfork m = 1$, where $\ell$ denotes the homology class of the longitude $S^1 \times \{ (1,0) \}$.
  
  On one hand, we have
  \[
   \bbSigma \times \bbI_{\gamma_\mp} = ((\bbSigma \times \bbI) \smallsetminus \bbN(\gamma_\mp)) \circ (\bbA_2 \disjun \id_{\bbSigma}).
  \]
  Indeed, we claim that the signature defect of the composition is given by
  \[
   -\mu(\lambda_m,\lambda_m,\lambda_{\ell \mp m}) = 0,
  \]
  where $\lambda_{\ell \mp m} \subset H_1(S^1 \times S^1;\R)$ is the Lagrangian subspace generated by $\ell \mp m$. To compute the first entry, remark that $\lambda_m$ is the kernel of the embedding induced by the outgoing boundary identification of the cobordism $S^1 \times \overline{D^2}$. To compute the third entry, remark that the pull-back of $\lambda$ to $H_1(S^1 \times S^1;\R)$ induced by the cobordism $(\Sigma \times I) \smallsetminus N(\gamma_\mp)$ contains $\ell \mp m$, because $\lambda$ contains the homology class of the curve $\gamma$ by hypothesis, which is homologous, in $(\Sigma \times I) \smallsetminus N(\gamma_\mp)$, to the longitude of $\partial N(\gamma_\mp)$ obtained as sum of the framing of $\gamma_\mp$ with the meridian of $N(\gamma_\mp)$, provided the latter is taken with sign $\pm 1$. Then, since Lagrangian subspaces of $H_1(S^1 \times S^1;\R)$ are 1-dimensional, this proves the claim.

  On the other hand, thanks to the isomorphism \eqref{E:mapping_cyl_eq_surgereg_cyl}, we have
  \[
   \bbSigma \times \bbI_{\tau_\gamma^{\pm 1}} = ((\Sigma \times I)(\gamma_\mp),P \times I,0),
  \]
  while
  \[
   ((\bbSigma \times \bbI) \smallsetminus \bbN(\gamma_\mp)) \circ (\bbB_2 \disjun \id_{\bbSigma}) = ((\Sigma \times I)(\gamma_\mp),P \times I,n_\mp),
  \]
  for some $n_\mp \in \Z$. Then, it follows from \cite[Prop.~4.10]{DGGPR19} that
  \begin{align*}
    \delta^{n_\mp} \rmV_\calC(\bbSigma \times \bbI_{\tau_\gamma^{\pm 1}}) 
    &= \rmV_\calC((\bbSigma \times \bbI) \smallsetminus \bbN(\gamma_\mp)) \circ 
    \rmV_\calC(\bbB_2 \disjun \id_{\bbSigma}) \\*
    &= \calD^{-1} \rmV_\calC((\bbSigma \times \bbI) \smallsetminus \bbN(\gamma_\mp)) \circ 
    \rmV_\calC(\bbA_2 \disjun \id_{\bbSigma}) \\*
    &= \calD^{-1} \rmV_\calC(\bbSigma \times \bbI_{\gamma_\mp}).
  \end{align*}

  It remains to show that $\Delta_\mp = \calD \delta^{n_\mp}$, i.e. that $n_\mp = \mp 1$. We claim that
  \[
   n_\mp = -\mu(\lambda_\ell,\lambda_m,\lambda_{\ell \mp m})
  \]
  where $\lambda_\ell \subset H_1(S^1 \times S^1;\R)$ is the Lagrangian subspace generated by $\ell$. To compute the first entry, remark that $\lambda_\ell$ is the kernel of the embedding induced by the outgoing boundary identification of the cobordism $D^2 \times S^1$. Now we have $(\lambda_\ell + \lambda_m) \cap \lambda_{\ell \mp m} = \lambda_{\ell \mp m}$, and
  \[
   \langle \ell \mp m,\ell \mp m \rangle = (\mp m) \pitchfork (\ell \mp m) = \mp m \pitchfork \ell = \pm 1.
  \]
  This means
  \[
   n_\mp = \mp 1. \qedhere
  \]
\end{proof}

\section{Equivalence for Hopf algebras}\label{A:equivalence}

In this appendix, we consider $\calC = \mods{H}$ for a finite-di\-men\-sion\-al factorizable ribbon Hopf algebra $H$ over $\Bbbk$, and we prove that the construction of \cite{DGGPR19} agrees with the one of \cite{DGP17}. More precisely, let $\calR_\lambda$ denote the category of bichrome graphs defined in \cite[Sec.~2.2]{DGP17}, and let us consider the functor 
\[
 F_\calR : \calR_\intL \to \calR_\lambda
\]
which preserves blue edges and coupons, labels red edges by the regular representation $H$, and replaces bichrome coupons, as recalled in Section~\ref{sss:LRT-functor}, with those defined in \cite[Sec.~2.2]{DGP17} like
\[
 \pic{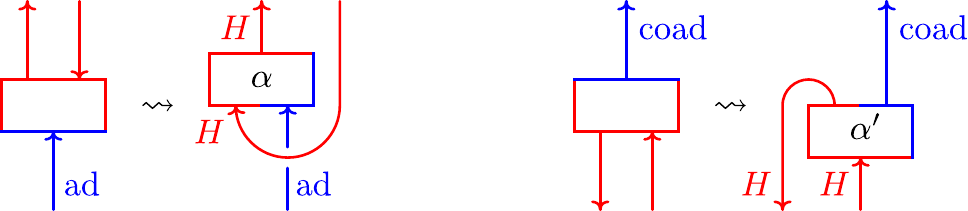}
\]
where $\alpha \in [1]\calC([1]\ad,[1]\one)$ and $\alpha' \in [1]\calC([1]\one,[1]\coad)$ are defined as
\[
 \alpha := \pic{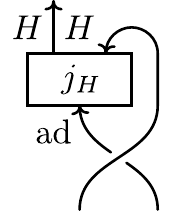} \qquad \qquad
 \alpha' := \pic{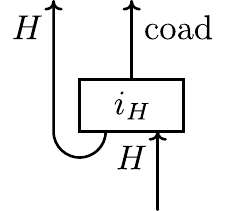}
\]
Next, let $F_\lambda : \calR_\lambda \to \calC$ denote the Hennings-Reshetikhin-Turaev functor defined in \cite[Sec.~2.2]{DGP17}.

\begin{proposition}\label{P:equivalence_H-L}
 The diagram
 \begin{center}
 \begin{tikzpicture}[descr/.style={fill=white}]
  \node (P1) at (165:2.25) {$\calR_\intL$};
  \node (P2) at (195:2.25) {$\calR_\lambda$};
  \node (P3) at (0:0) {$\calC \hspace*{10pt}$};
  \draw
  (P1) edge[->] node[left] {\scriptsize $F_\calR$} (P2)
  (P1) edge[->] node[above] {\scriptsize $F_\intL$} (P3)
  (P2) edge[->] node[below] {\scriptsize $F_\lambda$} (P3);
 \end{tikzpicture}
 \end{center}
 commutes.
\end{proposition}

\begin{proof}
The commutativity of the diagram is clear on objects. We check it now on morphisms.
 Let us consider a ribbon graph $T : (\myuline{\epsilon},\myuline{V}) \to (\myuline{\epsilon'},\myuline{V'})$ with $n$ red components in $\calR_\intL$. Thanks to Lemma~\ref{L:comparison_H-L}, we have
 \begin{equation}\label{eq:F-Lambda-H}
  F_\intL(T) = F_\calC(\tilde{T}_{(H,\ldots,H)}) \circ \left( (\lambda \otimes 1)^{\otimes n} \otimes \id_{F_\calC(\myuline{\epsilon},\myuline{V})} \right),
 \end{equation}
 where $\tilde{T}$ is an $n$-bottom graph presentation of $T$, where $\tilde{T}_{(H,\ldots,H)}$ is obtained from $\tilde{T}$ by labeling all its red edges by the regular representation $H$, and where $F_\calC$ denotes the Reshetikhin-Turaev functor, which can be evaluated against $\tilde{T}_{(H,\ldots,H)}$, forgetting the difference between red and blue. On the other hand, we can turn the $n$-bottom graph presentation $\tilde{T}$ into an $n$-string link graph presentation as follows
 \[
  \pic{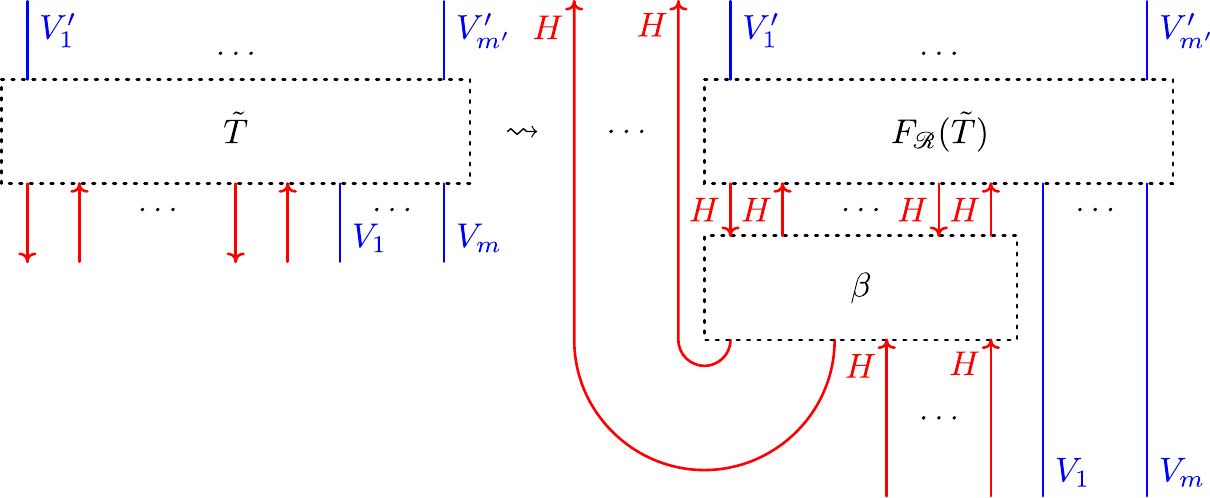}
 \]
 where $\beta$ is the $H$-labeled red $2n$-braid
 \[
  \beta := \pic{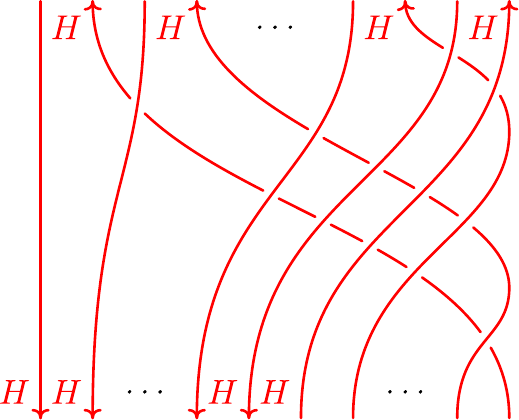}
 \]
 Remark that, strictly speaking, $\tilde{T}$ is not a morphism of $\calR_\intL$, and so $F_\calR(\tilde{T})$ is not actually defined. What we mean by this notation is that $F_\calR(\tilde{T})$ is obtained from $\tilde{T}$ by following the procedure for the definition of $F_\calR$. Now
 \begin{align}\label{eq:F-lambda-H}
  F_\lambda(F_\calR(T)) 
  &= F_\calC(F_\calR(\tilde{T})) \circ \left( \left( F_\calC(\beta) \circ (\lambda^{\otimes n} \otimes 1^{\otimes n}) \right) \otimes \id_{F_\calC(\myuline{\epsilon},\myuline{V})} \right),
 \end{align}
 where again $F_\calC$ can be evaluated against $F_\calR(\tilde{T})$, forgetting the difference between red and blue. 
 We now compare this expression with~\eqref{eq:F-Lambda-H}. 
 First of all, the fact that 
 \[
  F_\calC(F_\calR(\tilde{T})) = F_\calC(\tilde{T}_{(H,\ldots,H)})
 \]
 is clear from the definition of $F_\calR$. Furthermore, we claim the braid $\beta$ contributes trivially to the computation of $F_\lambda(F_\calR(T))$. Indeed, this is a direct consequence of the following general argument: for all objects $U,V,W \in \calC$ let us denote with 
 \[
  U \otimes V \otimes \_^* \otimes \_ \otimes W : \calC^\op \times \calC \to \calC
 \]
 the functor sending every object $(X,Y)$ of $\calC^\op \times \calC$ to the object $U \otimes V \otimes X^* \otimes Y \otimes W$ of $\calC$. Let $Z \in \calC$ be an object and 
 \[
  \alpha : U \otimes V \otimes \_^* \otimes \_ \otimes W \din Z
 \]
 be a dinatural transformation. Then, let us consider the dinatural transformation 
 \[
  \alpha_c : U \otimes \_^* \otimes \_ \otimes V \otimes W \din Z
 \]
 whose component $\alpha_{c,X} : U \otimes X^* \otimes X \otimes V \otimes W \to Z$ is given by
 \[
  \pic{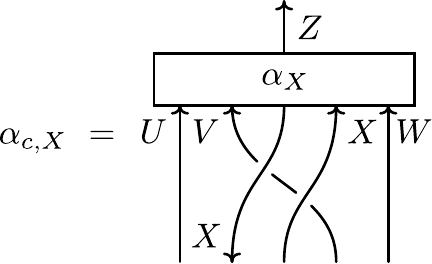}
 \]
 for every object $X \in \calC$. Then, for all $u \in U$, $v \in V$, and $w \in W$, we have
 \begin{align*}
  \alpha_{c,H} \left( u \otimes \lambda \otimes 1 \otimes v \otimes w \right) 
  &= \alpha_H \left( u \otimes R'' \cdot v \otimes \lambda(S(R'_{(1)}) \_ ) \otimes R'_{(2)} \otimes w \right) \\*
  &= \alpha_H \left( u \otimes R'' \cdot v \otimes \lambda(S(R'_{(1)}) \_ R'_{(2)}) \otimes 1 \otimes w \right) \\*
  &= \alpha_H \left( u \otimes R'' \cdot v \otimes \lambda(S^2(R'_{(2)}) S(R'_{(1)}) \_ ) \otimes 1 \otimes w \right) \\*
  &= \alpha_H \left( u \otimes R'' \cdot v \otimes \lambda(S(R'_{(1)} S(R'_{(2)})) \_ ) \otimes 1 \otimes w \right) \\*
  &= \alpha_H \left( u \otimes v \otimes \lambda \otimes 1 \otimes w \right),
 \end{align*}
 where the second equality follows from the fact that $\alpha$ is dinatural, and the third and fourth ones from the identities 
 \[
  \lambda(xy) = \lambda(S^2(y)x), \qquad
  R'_{(1)} S(R'_{(2)}) \otimes R'' = 1 \otimes 1
 \]
 respectively, which hold for every $x,y \in H$. This argument allows us to ignore all crossings belonging to the braid $\beta$, and it can be applied one red component at a time. Start from the middle one with respect to the source of the braid $\beta$, then move on to the one embracing it, and so on until all crossings have been considered.
We thus conclude that $F_\intL(T)$ from~\eqref{eq:F-Lambda-H} indeed equals $F_\lambda(F_\calR(T))$ in~\eqref{eq:F-lambda-H}.
\end{proof}

Finally, if we temporarily adopt the notation $\adCob_\calC^\rmL$ for the admissible cobordism category recalled in Section~\ref{sss:admissible_cobord}, and if $\adCob_\calC^\rmH$ denotes the one defined in \cite[Sec.~3.3]{DGP17}, then the functor $F_\calR : \calR_\intL \to \calR_\lambda$ induces a functor 
\[
 F_\Cob : \adCob_\calC^\rmL \to \adCob_\calC^\rmH.
\]
Let us also temporarily adopt the notation $\rmV_\calC^\rmL$ for the TQFT recalled in Section~\ref{sss:TQFT-from-universal}, and let $\rmV_\calC^\rmH$ denote the one defined in \cite[Sec.~3.3]{DGP17}

\begin{proposition}
 The diagram
 \begin{center}
 \begin{tikzpicture}[descr/.style={fill=white}]
  \node (P1) at (165:2.25) {$\adCob_\calC^\rmL$};
  \node (P2) at (195:2.25) {$\adCob_\calC^\rmH$};
  \node (P3) at (0:0) {$\Vect_\Bbbk$};
  \draw
  (P1) edge[->] node[left] {\scriptsize $F_\Cob$} (P2)
  (P1) edge[->] node[above] {\scriptsize $\rmV_\calC^\rmL$} (P3)
  (P2) edge[->] node[below] {\scriptsize $\rmV_\calC^\rmH$} (P3);
 \end{tikzpicture}
 \end{center}
 commutes.
\end{proposition}

\begin{proof}
 The claim follows from the explicit identification of state spaces of \cite[Sec.~4.7]{DGGPR19}, which uses a slightly different convention with respect to the one recalled in Section~\ref{sss:skein-alg-summary}, and from the one of \cite[Sec.~3.7]{DGP17}. Remark that the latter uses the \textit{pre-dual coadjoint representation} $X$ instead of $\ad$. As a vector space, $X$ coincides with $H$ itself, just like $\ad$, but it carries the left $H$-action defined by
 \[
  x \cdot y := x_{(2)}yS^{-1}(x_{(1)})
 \]
 for all $x,y \in H$, see \cite[Sec.~3.1]{DGP17}. However, the inverse antipode $S^{-1} : \ad \to X$ defines an isomorphism between the two representations.
 Then, in order to finish the proof, we just need to relate 
 the pairing $\langle \_,\_ \rangle_{g,V}^\rmL$ of \cite[Sec.~4.1]{DGGPR19} to the pairing $\langle \_,\_ \rangle_{g,V}^\rmH$ of \cite[Sec.~3.1]{DGP17}. Let us consider the linear isomorphism
 \[
  F_\calX : \calC(P_{\one},\ad^{\otimes g} \otimes V) \to \calC(H,X^{\otimes g} \otimes V)
 \]
 which sends $x \in \calC(P_{\one},\ad^{\otimes g} \otimes V)$ to 
 \[
  ((S^{-1} \circ \rho^{-1} \circ D^{-1} \circ \rho^{-1} \circ D^{-1})^{\otimes g} \otimes \id_V) \circ x \circ \pi_{\one} \in \calC(H,X^{\otimes g} \otimes V),
 \]
 where $\rho : \ad \to \coad$ is the \textit{Radford map} sending $x \in \ad$ to $\lambda(S(x) \_) \in \coad$, and where $\pi_{\one} : H \to P_{\one}$ is the projection onto the indecomposable summand containing the cointegral $\lambda^\co \in H$. Remark that the isomorphism $h_X : X \to \coad$ used in \cite[Sec.~3.1]{DGP17} to define $\langle \_,\_ \rangle_{g,V}^\rmH$ is given by $\rho \circ D \circ \rho \circ S$, and that the injective envelope morphism $\eta_{\one} : \one \to P_{\one}$ satisfies $\eta_{\one}(1) = \pi_{\one}(\lambda^\co)$. Then, we clearly have
 \[
  \langle x',x \rangle_{g,V}^\rmL = \langle x',F_\calX(x) \rangle_{g,V}^\rmH. \qedhere
 \]
\end{proof}

\addtocontents{toc}{\protect\setcounter{tocdepth}{2}}


\begin{thebibliography}{AAAAA}
 
 \bibitem[BBGa]{BBG18}
 A.~Beliakova, C.~Blanchet, A.~Gainutdinov,
 \textit{Modified Trace is a Symmetrised Integral},
 to appear in 
 \doi{10.1007/s00029-021-00626-5}{Selecta Math.};
 \arxiv{1801.00321}{[math.QA]}.
 
 \bibitem[BBGe]{BBG17}
 A.~Beliakova, C.~Blanchet, N.~Geer,
 \textit{Logarithmic Hennings Invariants for Restricted Quantum $\mathfrak{sl}(2)$},
 \doi{10.2140/agt.2018.18.4329}{Algebr. Geom. Topol. \textbf{18} (2018), No. 7, 4329--4358};
 \arxiv{1705.03083}{[math.GT]}.

 \bibitem[BCGP]{BCGP14}
 C.~Blanchet, F.~Costantino, N.~Geer, B.~Patureau-Mirand,  
 \textit{Non-Semisimple TQFTs, Reidemeister Torsion and Kashaev's Invariants}, 
 \doi{10.1016/j.aim.2016.06.003}{Adv. Math. \textbf{301} (2016), 1--78}; \arxiv{1404.7289}{[math.GT]}.
 
 \bibitem[BHMV]{BHMV95}
 C.~Blanchet, N.~Habegger, G.~Masbaum, P.~Vogel,
 \textit{Topological Quantum Field Theories Derived from the Kauffman Bracket},
 \doi{10.1016/0040-9383(94)00051-4}{Topology \textbf{34} (1995), No. 4, 883--927}.
 
 \bibitem[CGP]{CGP12}
 F.~Costantino, N.~Geer, B.~Patureau-Mirand,  
 \textit{Quantum Invariants of 3-Manifolds via Link Surgery Presentations and Non-Semi-Simple Categories}, \doi{10.1112/jtopol/jtu006}{J. Topol. \textbf{7} (2014), No. 4, 1005--1053};
 \arxiv{1202.3553}{[math.GT]}. 

 \bibitem[DGGPR]{DGGPR19}
 M.~De~Renzi, A.~Gainutdinov, N.~Geer, B.~Patureau-Mirand, I.~Runkel,
 \textit{3-Dimensional TQFTs from Non-Semisimple Modular Categories}, 
 \arxiv{1912.02063}{[math.GT]}.
  
 \bibitem[DGP]{DGP17}
 M.~De Renzi, N.~Geer, B.~Patureau-Mirand,
 \textit{Renormalized Hennings Invariants and 2+1-TQFTs},
 \doi{10.1007/s00220-018-3187-8}{Comm. Math. Phys. \textbf{362} (2018), No. 3, 855--907};
 \arxiv{1707.08044}{[math.GT]}.

 \bibitem[Dr]{D90}
 V.~Drinfel'd,
 \textit{Almost Cocommutative Hopf Algebras},
 (Russian) \href{http://www.mathnet.ru/php/archive.phtml?wshow=paper&jrnid=aa&paperid=10&option_lang=eng}{Algebra i Analiz \textbf{1} (1989), No. 2, 30--46};
 translation in Leningrad Math. J. \textbf{1} (1990), No. 2, 321--342.

 \bibitem[EGNO]{EGNO15}
 P.~Etingof, S.~Gelaki, D.~Nikshych, V.~Ostrik,
 \textit{Tensor Categories}, 
 \doi{10.1090/surv/205}{Math. Surveys Monogr. \textbf{205}, Amer. Math. Soc., Providence, RI, 2015}.

 \bibitem[Fa]{F18}
 M.~Faitg,
 \textit{Projective Representations of Mapping Class Groups in Combinatorial Quantization},
 \doi{10.1007/s00220-019-03470-z}{Comm. Math. Phys. \textbf{377} (2020), no.~1, 161--198};
 \arxiv{1812.00446}{[math.QA]}.

 \bibitem[FGR]{FGR17}
 V.~Farsad, A.~Gainutdinov, I.~Runkel,
 \textit{$\SL(2,\Z)$-Action for Ribbon Quasi-Hopf Algebras},
 \doi{10.1016/j.jalgebra.2018.12.012}{J. Algebra \textbf{522} (2019), 243--308};
 \arxiv{1702.01086}{[math.QA]}.

 \bibitem[FGST]{FGST05} 
 B.~Feigin, A.~Gainutdinov, A.~Semikhatov, I.~Tipunin, 
 \textit{Modular Group Representations and Fusion in Logarithmic Conformal Field Theories and in the Quantum Group Center},
 \doi{10.1007/s00220-006-1551-6}{Comm. Math. Phys. \textbf{265} (2006), No. 1, 47--93};
 \href{https://arxiv.org/abs/hep-th/0504093}{\texttt{arXiv:hep-th/0504093}}.

 \bibitem[FM]{FM12}
 B.~Farb, D.~Margalit,
 \textit{A Primer on Mapping Class Groups},
 \doi{10.1515/9781400839049}{Princeton Mathematical Series \textbf{49}, Princeton University Press, Princeton, NJ, 2012}.

 \bibitem[FSS]{FSS12}
 J.~Fuchs, C.~Schweigert, C.~Stigner,
 \textit{Higher Genus Mapping Class Group Invariants From Factorizable Hopf Algebras},
 \doi{10.1016/j.aim.2013.09.019}{Adv. Math. \textbf{250} (2014), 285--319};
 \arxiv{1207.6863}{[math.QA]}.
 
 \bibitem[GKP1]{GKP10}
 N.~Geer, J.~Kujawa, B.~Patureau-Mirand,
 \textit{Generalized Trace and Modified Dimension Functions on Ribbon Categories},
 \doi{10.1007/s00029-010-0046-7}{Selecta Math. \textbf{17} (2011), No. 2, 435--504};
 \arxiv{1001.0985}{[math.RT]}.

 \bibitem[GKP2]{GKP11} 
 N.~Geer, J.~Kujawa, B.~Patureau-Mirand,
 \textit{Ambidextrous Objects and Trace Functions for Nonsemisimple Categories},
 \doi{10.1090/S0002-9939-2013-11563-7}{Proc. Amer. Math. Soc. \textbf{141} (2013), No. 9, 2963--2978};
 \arxiv{1106.4477}{[math.RT]}.
  
 \bibitem[GKP3]{GKP18} 
 N.~Geer, J.~Kujawa, B.~Patureau-Mirand,
 \textit{M-Traces in (Non-Unimodular) Pivotal Categories},
 \arxiv{1809.00499}{[math.RT]}.
 
 \bibitem[GPT]{GPT07}
 N.~Geer, B.~Patureau-Mirand, V.~Turaev,
 \textit{Modified Quantum Dimensions and Re-Nor\-mal\-ized Link Invariants},
 \doi{10.1112/S0010437X08003795}{Compos. Math. \textbf{145} (2009), No. 1, 196--212};
 \arxiv{0711.4229}{[math.QA]}.
 
 \bibitem[GR]{GR17}
 A.~Gainutdinov, I.~Runkel,
 \textit{Projective Objects and the Modified Trace in Factorisable Finite Tensor Categories},
 \doi{10.1112/S0010437X20007034}{Compos. Math. \textbf{156} (2020), No. 4, 770--821};
 \arxiv{1703.00150}{[math.QA]}.

 \bibitem[He]{H96}
 M.~Hennings,
 \textit{Invariants of Links and 3-Manifolds Obtained from Hopf Algebras},
 \doi{10.1112/jlms/54.3.594}{J. London Math. Soc. (2) \textbf{54} (1996), No. 3, 594--624}.
 
 \bibitem[Ka]{K95}
 C.~Kassel,
 \textit{Quantum Groups},
 \doi{10.1007/978-1-4612-0783-2}{Graduate Texts in Mathematics \textbf{155}, Springer-Verlag, New York, 1995}.
 
 \bibitem[Ke1]{K94}
 T.~Kerler,
 \textit{Mapping Class Group Actions on Quantum Doubles},
 \doi{10.1007/BF02101554}{Comm. Math. Phys. \textbf{168} (1995), No. 2, 353--388};
 \href{http://arXiv.org/abs/hep-th/9402017}{\texttt{arXiv:hep-th/9402017}}
 
 \bibitem[Ke2]{K96}
 T.~Kerler,
 \textit{Genealogy of Nonperturbative Quantum-Invariants of 3-Manifolds: The Surgical Family}, in 
 \href{https://www.crcpress.com/Geometry-and-Physics/Pedersen-Andersen-Dupont-Swann/p/book/9780824797911}{Geometry and Physics (Aarhus, 1955), 503--547, Lecture Notes in Pure and Appl. Math. \textbf{184}, Dekker, New York, 1997};
 \href{http://arXiv.org/abs/q-alg/9601021}{\texttt{arXiv:q-alg/9601021}}.

 \bibitem[KL]{KL01}
 T.~Kerler, V.~Lyubashenko,
 \textit{Non-Semisimple Topological Quantum Field Theories for 3-Manifolds with Corners},
 \doi{10.1007/b82618}{Lecture Notes in Mathematics \textbf{1765}. Springer-Verlag, Berlin, 2001}.

 \bibitem[Li1]{L62}
 W.~Lickorish,
 \textit{A Representation of Orientable Combinatorial 3-Manifolds},
 \doi{10.2307/1970373}{Ann. of Math. (2) \textbf{76} (1962), No. 3, 531--540}.
 
 \bibitem[Li2]{L64}
 W.~Lickorish,
 \textit{A Finite Set of Generators for the Homeotopy Group of a 2-Manifold},
 \doi{10.1017/S030500410003824X}{Proc. Cambridge Philos. Soc. \textbf{60} (1964), No. 4, 769--778}.
 
 \bibitem[Lu]{L90}
 G.~Lusztig,
 \textit{Finite Dimensional Hopf Algebras Arising from Quantized Universal Enveloping Algebras},
 \doi{10.1090/S0894-0347-1990-1013053-9}{J. Amer. Math. Soc. \textbf{3} (1990), No. 1, 257--296}.
 
 \bibitem[Ly1]{L94} 
 V.~Lyubashenko,
 \textit{Invariants of 3-Manifolds and Projective Representations of Mapping Class Groups via Quantum Groups at Roots of Unity},
 \doi{10.1007/BF02101805}{Comm. Math. Phys. \textbf{172} (1995), No. 3, 467--516};
 \href{http://arXiv.org/abs/hep-th/9405167}{\texttt{arXiv:hep-th/9405167}}.
 
 \bibitem[Ly2]{L95} 
 V.~Lyubashenko,
 \textit{Modular Transformations for Tensor Categories},
 \doi{10.1016/0022-4049(94)00045-K}{J. Pure Appl. Algebra \textbf{98} (1995), No. 3, 279--327}.
 
 \bibitem[Ly3]{L96} 
 V.~Lyubashenko,
 \textit{Ribbon Abelian Categories as Modular Categories},
 \doi{10.1142/S0218216596000229}{J. Knot Theory Ramifications \textbf{05} (1996), No. 3, 311--403}.
 
 \bibitem[LM]{LM94}
 V.~Lyubashenko, S.~Majid,
 \textit{Braided Groups and Quantum Fourier Transform},
 \doi{10.1006/jabr.1994.1165}{J. Algebra \textbf{166} (1994), no.~3, 506--528}.

 \bibitem[Ma1]{M93}
 S.~Majid,
 \textit{Braided Groups},
 \doi{10.1016/0022-4049(93)90103-Z}{J. Pure Appl. Algebra \textbf{86} (1993), No. 2, 187--221}.
 
 \bibitem[Ma2]{M95}
 S.~Majid, 
 \textit{Foundations of Quantum Group Theory},
 \doi{10.1017/CBO9780511613104}{Cambridge University Press, 1995}.
  
 \bibitem[ML]{M71}
 S.~Mac Lane, 
 \textit{Categories for the Working Mathematician},
 \doi{10.1007/978-1-4612-9839-7}{Graduate Texts in Mathematics \textbf{5}, Springer-Verlag, New York-Berlin, 1971}.
 
 \bibitem[Mu]{M13} 
 J.~Murakami,  
 \textit{Generalized Kashaev Invariants for Knots in Three Manifolds},  
 \doi{10.4171/QT/86}{Quantum Topol. \textbf{8} (2017), No. 1, 35--73};
 \arxiv{1312.0330}{[math.GT]}.
 
 \bibitem[Os]{O95}
 V.~Ostrik,
 \textit{Decomposition of the Adjoint Representation of the Small Quantum $\sl_2$},
 \doi{10.1007/s002200050109}{Comm. Math. Phys. \textbf{186} (1997), No. 2, 253--264}; 
 \href{https://arxiv.org/abs/q-alg/9512026}{\texttt{arXiv:q-alg/9512026}}.
 
 \bibitem[Ra]{R12}
 D.~Radford,
 \textit{Hopf Algebras},
 \doi{10.1142/8055}{Series on Knots and Everything \textbf{49}, World Scientific Publishing Co. Pte. Ltd., Hackensack, NJ, 2012.}

 \bibitem[RT1]{RT90}
 N.~Reshetikhin, V.~Turaev,
 \textit{Ribbon Graphs and Their Invariants Derived From Quantum Groups},
 \doi{10.1007/BF02096491}{Comm. Math. Phys. \textbf{127} (1990), No. 1, 1--26}.

 \bibitem[RT2]{RT91}
 N.~Reshetikhin, V.~Turaev,
 \textit{Invariants of 3-Manifolds via Link Polynomials and Quantum Groups},
 \doi{10.1007/BF01239527}{Invent. Math. \textbf{103} (1991), No. 1, 547--597}.
 
 \bibitem[Sh]{S16}
 K.~Shimizu,
 \textit{Non-Degeneracy Conditions for Braided Finite Tensor Categories},
 \doi{10.1016/j.aim.2019.106778}{Adv. Math. \textbf{355} (2019), 106778};
 \arxiv{1602.06534}{[math.QA]}.

 \bibitem[Tu]{T94} 
 V.~Turaev,
 \textit{Quantum Invariants of Knots and 3-Manifolds},
 \href{https://www.degruyter.com/view/product/461906}{De Gruyter Studies in Mathematics \textbf{18}, Walter de Gruyter \& Co., Berlin, 1994}.

\end{thebibliography}
\end{document}